\documentclass{article}
\usepackage[utf8]{inputenc}

\usepackage[margin=1in]{geometry}
\usepackage{amsthm, amssymb, amsmath,graphicx, mathtools, mathrsfs, bbm, tikz, tikz-network, caption, subcaption, comment}

\usepackage[
backend=biber,
style=numeric,
sorting=nyt,
sortcites=true
]{biblatex}

\newcommand{\cN}{\mathcal{N}}
\newcommand{\bR}{\mathbb{R}}
\newcommand{\cE}{\mathcal{E}}

\newcommand{\gam}[1]{\Gamma_{\text{#1}}}

\newcommand{\adm}{\text{Adm}}
\newcommand{\convt}{\text{co}}
\newcommand{\dom}{\text{Dom}}
\newcommand{\ext}{\text{ext}}
\newcommand{\Mod}{\text{Mod}}

\DeclareMathOperator{\supp}{supp}

\newcommand{\rv}[1]{\underline{#1}}
\newcommand{\bfec}{\Gamma_{\overline{\text{fec}}}}
\newcommand{\modeq}{\simeq}
\DeclareMathOperator*{\argmin}{arg\,min}

\newtheorem{theorem}{Theorem}[section]
\newtheorem{corollary}[theorem]{Corollary}
\newtheorem{lemma}[theorem]{Lemma}
\newtheorem{proposition}[theorem]{Proposition}

\theoremstyle{definition}
\newtheorem{definition}[theorem]{Definition}

\theoremstyle{remark}
\newtheorem{remark}{Remark}
\newtheorem{example}{Example}

\allowdisplaybreaks

\title{Modulus of edge covers and stars}
\author{Adriana Ortiz-Aquino\thanks{\protect\url{adrianaortiz@ksu.edu}}\ \ and Nathan Albin\thanks{\protect\url{albin@k-state.edu}}}
\date{
{\small Department of Mathematics, Kansas State University, Manhattan, KS 66506, USA. }\\[2ex]
February 29, 2024
}
\begin{document}

\maketitle

\begin{abstract}
This paper explores the modulus (discrete $p$-modulus) of the family of edge covers on a discrete graph.  This modulus is closely related to that of the larger family of fractional edge covers; the modulus of the latter family is guaranteed to approximate the modulus of the former within a multiplicative factor. The bounds on edge cover modulus can be computed efficiently using a duality result that relates the fractional edge covers to the family of stars.
\end{abstract}

\section{Introduction}\label{sec:introduction}

The discrete $p$-modulus is a versatile and informative tool for studying many families of objects on graphs~\cite{albin2017modulus, albin2015modulus}. For example, the modulus of all paths connecting two nodes in a graph is related to known graph theoretic quantities such as shortest path, max flow/min cut, and effective resistance~\cite{albin2017modulus,shakeri2016generalized}. It has also been shown that the modulus of all cycles can be used for clustering and community detection~\cite{shakeri2017network} and that the modulus of all spanning trees can be used to describe a hierarchical decomposition of the graph~\cite{albin2020spanning,kottegoda2020spanning,albin2021fairest}. In general, modulus can be adapted to any family of graph objects and will measure the richness of that family; larger families of objects tend to have larger modulus. This paper considers the modulus of edge covers and gives an approximation to its value using the modulus of the family of fractional edge covers. 

As described in Section~\ref{sec:defs} in more detail, the modulus problem is a convex optimization problem with a number of constraints determined by the number of elements in the family of graph objects under consideration. Thus, the modulus of a combinatorially large family, like the family of edge covers, can be difficult to compute directly. Through the theory of Fulkerson duality, it has been shown that every family of objects has a corresponding dual family, whose modulus is closely related to the modulus of the original family \cite{fulkerson1968blocking, albin2019blocking}. We prove in this paper that the dual family of fractional edge covers is the family of stars, which greatly reduces the number of constraints for the $p$-modulus problem. In this way, we can calculate the modulus of fractional edge covers using the modulus of stars, and then obtain a bound for the modulus of edge covers.

A probabilistic interpretation for the $p$-modulus problem has also proven valuable in some cases~\cite{albin2016prob, albin2021convergence}. This allows us to reinterpret the modulus problem as an optimization problem related to random objects. Specifically, it has been shown that solving the modulus problem is equivalent to finding a probability mass function (pmf) on the family of objects that minimizes a function of certain expectations on the edges. 

The primary contributions of this paper are the following.
\begin{itemize}
    \item Section~\ref{sec:defs} introduces a new equivalence relation of families of graph objects with respect to modulus. 
    \item Definition~\ref{def:bfec} introduces the concept of basic fractional edge covers.
    \item Lemma~\ref{lem:extreme-bfec} characterizes the set of extreme points of fractional edge covers by relating them to basic fractional edge covers. 
    \item Theorem~\ref{thm:bounds} provides an estimate of the modulus of edge covers using the modulus of fractional edge covers. 
    \item Lemma~\ref{lem:Bhat_fec} shows that the dual blocking family of fractional edge covers is equivalent to the family of stars, which allows the modulus of fractional edge covers to be computed more efficiently.
\end{itemize}

This paper is organized as follows. Section \ref{sec:defs} introduces definitions and notations. Section \ref{sec:mod_ec_fec} defines the modulus problem for the family of edge covers and the family of fractional edge covers. Section \ref{sec:fulkerson_duality} reviews the concepts of blocking duality and proves that the dual family of fractional edge covers is the family of stars. Section \ref{sec:star_mod} explores the star modulus problem, reviews the probabilistic interpretation of modulus, and demonstrates these concepts through examples on several standard graphs. Section \ref{sec:num-examples} explores some of the fundamental differences between edge cover modulus and fractional edge cover modulus through two additional examples. Section \ref{sec:discussion} ends with a discussion of these concepts and future work. 

\section{Definitions and Notation}\label{sec:defs}

\paragraph{Objects and usage.} Let $G = (V,E,\sigma)$ be an undirected graph with vertex set $V$, edge set $E$, and a positive vector $\sigma\in\mathbb{R}^E_{>0}$ that assigns to each edge, $e\in E$, a positive weight, $\sigma(e)$. Let $\Gamma$ be a \emph{family of objects} on $G$. The concept of \emph{object} is very flexible. The present paper is focused primarily on three specific families: the families of stars, edge covers, and fractional edge covers. Each of these families is defined below. 

The family $\Gamma$ is associated with a nonnegative \emph{usage matrix}, $\mathcal{N}\in\mathbb{R}_{\ge 0}^{\Gamma\times E}$, where $\mathcal{N}(\gamma,e)$ indicates the degree to which the object $\gamma\in\Gamma$ ``uses'' the edge $e\in E$. When $\Gamma$ consists of subsets of $E$, a natural choice for $\mathcal{N}$ is the indicator function
\begin{equation}\label{eq:natural-N}
    \mathcal{N}(\gamma,e) = \mathbbm{1}_{\gamma}(e) := 
    \begin{cases} 1 & \text{if }e \in \gamma, \\ 0 & \text{otherwise}. \end{cases}
\end{equation}
For the family of fractional edge covers, however, $\mathcal{N}$ will be allowed to take other real values. In this paper, we restrict our attention to families that are \emph{nontrivial} in the sense that each row of $\mathcal{N}$ contains at least one nonzero entry. Note that we refer to $\mathcal{N}$ as a matrix, even if $\Gamma$ is an infinite set.

The usage matrix provides a useful representation of the objects in $\gamma$; by associating each $\gamma$ with the corresponding row vector, $\mathcal{N}(\gamma,\cdot)$ in the usage matrix, we may view a family of objects as a subset of the nonnegative vectors $\mathbb{R}^E_{\ge 0}$. Given a family of objects, $\Gamma\subseteq\mathbb{R}^E_{\ge 0}$, it is often useful to consider its \emph{convex hull}, $\convt(\Gamma)$, as well as its \emph{dominant},
\begin{equation*}
    \dom(\Gamma) = \convt(\Gamma) + \mathbb{R}^{E}_{\ge 0}.
\end{equation*}

\paragraph{Stars, edge covers, and fractional edge covers.}

To each vertex, $v\in V$, is associated the \emph{star}, $\delta(v)\subset E$, comprising the set of edges incident to $v$. The family, $\gam{star}$, of all stars in $G$ is endowed with the natural usage matrix~\eqref{eq:natural-N}. Any subset of a star is called a \emph{substar}.

An \emph{edge cover} of $G$ is a set of edges $C\subset E$ such that each vertex in $G$ is incident to at least one edge in $C$, that is $|C\cap\delta(v)|\ge 1$ for every $v\in V$. (If the intersection is exactly 1 for all vertices, the edge cover is called a \emph{perfect matching}.) The family of all edge covers is denoted $\gam{ec}$ and is also endowed with the natural usage matrix~\eqref{eq:natural-N}.

The concept of edge cover can be generalized as follows (see~\cite{schrijver2003combinatorial}). Let $\gamma\in\mathbb{R}_{\ge 0}^E$ be a nonnegative vector on $E$. If
\begin{equation*}
    \gamma(\delta(v)) :=
    \sum_{e\in\delta(v)}\gamma(e)\ge 1\quad\text{for all }v\in V,
\end{equation*}
then $\gamma$ is called a \emph{fractional edge cover}. Each such $\gamma$ can be considered to be an object on the graph $G$ with corresponding edge usage
\begin{equation*}
    \mathcal{N}(\gamma,\cdot) := \gamma(\cdot),
\end{equation*}
yielding the (uncountably infinite) family of fractional edge covers, $\gam{fec}$. Each edge cover, $\gamma\in\gam{ec}$, can be associated with its incidence vector $\mathbbm{1}_\gamma$, which provides a natural inclusion $\gam{ec}\subset\gam{fec}$. (The reverse inclusion is only true on the trivial graph.)

\paragraph{Densities and admissibility.}

A nonnegative vector, $\rho\in\mathbb{R}^E_{\ge 0}$, is called a \emph{density}. Each density induces a measure of length, called the \emph{$\rho$-length}, on the objects in $\Gamma$. Given a density $\rho$ and an object $\gamma\in\Gamma$, the $\rho$-length of $\gamma$, $\ell_\rho(\gamma)$, is defined as
\[
\ell_{\rho}(\gamma) := \sum_{e \in E} \cN(\gamma, e)\rho(e)= \rho^T\gamma,
\]
where the last equality arises from identifying $\gamma$ with its row in the matrix $\mathcal{N}$.

 A density, $\rho$, is called \emph{admissible for $\Gamma$}, or simply \emph{admissible}, if $ \ell_{\rho}(\gamma) \ge 1 $ for every $\gamma \in \Gamma$. The set of all admissible densities for a family $\Gamma$ is denoted as 
\[ 
\adm(\Gamma) := \{ \rho \in \bR_{\ge 0}^E : \ell_{\rho}(\gamma) \ge 1, \forall \gamma \in \Gamma\}.
\]

\begin{lemma}\label{lem:adm-monotone}
    Suppose $\Gamma_1$ and $\Gamma_2$ be two families of objects satisfying $\Gamma_1\subseteq\Gamma_2$, then $\adm(\Gamma_2)\subseteq\adm(\Gamma_1)$.
\end{lemma}
\begin{proof}
    This is evident from the definitions. Any density that is admissible for the larger family, $\Gamma_2$, is necessarily admissible for $\Gamma_1$.
\end{proof}

\paragraph{Energy, modulus, and extremal densities.} For $1\le p\le \infty$, the \emph{$p$-energy} of a density $\rho$ is defined as
\begin{equation*}
    \mathcal{E}_{p,\sigma}(\rho) :=
    \begin{cases}
        \sum\limits_{e\in E}\sigma(e)\rho(e)^p & \text{if }1\le p < \infty,\\
        \max\limits_{e\in E}\sigma(e)\rho(e) & \text{if }p=\infty.
    \end{cases}
\end{equation*}
The \emph{$p$-modulus} of $\Gamma$ is then defined as
\begin{equation}\label{eq:mod}
    \Mod_{p,\sigma}(\Gamma) := \inf_{\rho \in \adm(\Gamma)} \cE_{p,\sigma}(\rho).
\end{equation}
If the graph is unweighted, we define $\sigma \equiv 1$ and adopt the simplified notation $\cE_p(\rho)$ and $\Mod_p(\Gamma)$.

In the case that $\Gamma$ is finite, modulus can be expressed as a convex optimization problem of the form
\begin{equation}\label{eq:optimization_prob}
\begin{split}
    \text{minimize} & \quad \cE_{p,\sigma}(\rho) \\
    \text{subject to} & \quad \rho \succeq 0 \\
                      & \quad \cN \rho \succeq \mathbf{1}. \\ 
\end{split}
\end{equation}
The notation $\succeq$ indicates elementwise comparison and $\mathbf{1}$ indicates the appropriately shaped vector of all ones.

A density $\rho \in \adm(\Gamma)$ is said to be \emph{extremal} if $\cE_{p,\sigma}(\rho) = \Mod_{p,\sigma}(\Gamma)$. The notation $\rho^*$ is commonly used to denote an extremal density. If $\Gamma$ is finite, then \cite[Theorem 4.1]{albin2017modulus} implies that an extremal density exists. Moreover, it is unique for $1<p<\infty$. A useful property of modulus comes from Proposition 3.4 in \cite{albin2017modulus}, which is a direct consequence of Lemma~\ref{lem:adm-monotone}.

\begin{proposition}[Monotonicity]
    Let $\Gamma_1$ and $\Gamma_2$ be families of graph objects. If $\Gamma_1 \subseteq \Gamma_2$, then $\Mod_{p,\sigma}(\Gamma_1) \le \Mod_{p,\sigma}(\Gamma_2)$.
\end{proposition}

\paragraph{Equivalent families.} Two families of objects, $\Gamma$ and $\Gamma'$ are called \emph{equivalent} (in the sense of modulus) if $\adm(\Gamma)=\adm(\Gamma')$. We shall use the notation $\Gamma\modeq\Gamma'$ to indicate that the two families are equivalent in this sense. In light of~\eqref{eq:mod}, this implies that $\Mod_{p,\sigma}(\Gamma)=\Mod_{p,\sigma}(\Gamma')$ for any choice of the parameter $p$ and weights $\sigma$; equivalent families are indistinguishable in the context of modulus.

One straightforward example of equivalence comes from the following lemma.

\begin{lemma}\label{lem:gamma-equiv-dom}
    Let $\Gamma$ be a family of objects on $G$. Then $\Gamma\modeq\dom(\Gamma)$.
\end{lemma}

\begin{proof}
By definition, $\Gamma \subseteq \dom(\Gamma)$, so by Lemma~\ref{lem:adm-monotone}, $\adm(\dom(\Gamma)) \subseteq \adm(\Gamma)$. Let $\rho \in \adm(\Gamma)$ and let $\tilde{\gamma} \in \dom(\Gamma)$. Then there must exist a collection of objects, $\gamma_1,\gamma_2,\cdots,\gamma_r\in\Gamma$, a choice of weights $\mu_1, \mu_2, \ldots, \mu_r \ge 0$ summing to one, and a vector $\xi \in \bR^E_{\ge 0}$ such that
\[
\tilde{\gamma} = \sum_{i=1}^r \mu_i \gamma_i + \xi.
\]
Since $\rho$ and $\xi$ are nonnegative, and since $\rho$ is admissible for $\Gamma$,
\[
\rho^T \tilde{\gamma} = \sum_{i=1}^r \mu_i \rho^T \gamma_i + \rho^T \xi \ge 1.
\]
Thus, $\rho \in \adm(\dom(\Gamma))$. 
\end{proof}

For a given family $\Gamma$, consider the set of extreme points $\ext(\dom(\Gamma))$. The extreme points are nonnegative vectors in $\mathbb{R}^E$ and, therefore, can be viewed as a family of objects. Lemma~\ref{lem:gamma-equiv-dom} implies that $\Gamma\modeq\ext(\dom(\Gamma))$, since both families share the same dominant.

As an application of this last equivalence, consider the family, $\Gamma'_{st}$, of all walks in $G$ connecting two distinct vertices $s$ and $t$ and the family, $\Gamma_{st}$, of all simple paths connecting $s$ and $t$. Then, since $\Gamma_{st} = \ext(\dom(\Gamma'_{st})$, it follows that $\Gamma_{st}\modeq\Gamma'_{st}$. This simply recovers the intuitive observation that a density is admissible for the family of $st$-walks if and only if it is admissible for the family of $st$-paths.

\section{Modulus of edge covers and fractional edge covers}\label{sec:mod_ec_fec}

We first calculate $\Mod_2(\gam{ec})$ for some common graphs. The approach in each case is similar, we first establish an upper bound on the modulus by finding an admissible density. In later sections, we will develop lower bounds that are shown to coincide with the upper bounds of this section, thus establishing the value of the modulus.

\begin{figure}
\begin{subfigure}{0.3\textwidth}
    \centering
    \includegraphics[width=0.5\textwidth]{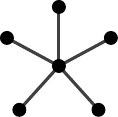}
    \caption{The Star Graph, $S_6$.}
    \label{fig:S6}
    \end{subfigure}
    \hfill
    \begin{subfigure}{0.3\textwidth}
    \centering
    \includegraphics[width=0.5\textwidth]{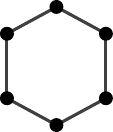}
    \caption{The Cycle Graph, $C_6$.}
    \label{fig:C6}
    \end{subfigure}
    \hfill
    \begin{subfigure}{0.3\textwidth}
    \centering
    \includegraphics[width=0.5\textwidth]{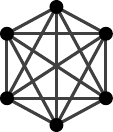}
    \caption{The Complete Graph, $K_6$.}
    \label{fig:K6}
\end{subfigure}
\caption{A standard set of graphs for modulus examples.}
\end{figure}

\begin{example}[Star Graph]\label{ex:star-ec}
Let $G = S_n$ be the unweighted star graph with $|V| = n\ge 3$ and $|E| = n-1$. Note that all edges of $S_n$ are \emph{pendant edges}---they are incident on at least one vertex of degree one. It follows that there is a single edge cover: $\gam{ec} = \{ E\}$. The symmetries of the graph and the uniqueness of the extremal density suggest that we restrict our search to constant densities. Since the single edge cover has $n-1$ edges, the density $\rho_0 \equiv \frac{1}{n-1}$ is admissible. This provides an upper bound on modulus: 
\[ 
\Mod_2(\gam{ec}) \le \cE_2(\rho_0) = \sum_{e \in E} \rho_0(e)^2 = (n-1) \cdot \frac{1}{(n-1)^2} = \frac{1}{n-1}.
\]
The fact that equality holds is established later through the lower bound in Example~\ref{ex:star-star}.
\end{example}

 Both of the next two examples attain a simple upper bound on the modulus.

 \begin{lemma}\label{lem:const-upper-bound-ec}
    Let $\gam{ec}$ be the family of edge covers on a graph, $G$, containing $n$ vertices and $m$ edges. Then
    \begin{equation*}
    \Mod_2(\gam{ec}) \le m\left\lceil\frac{n}{2}\right\rceil^{-2}.
    \end{equation*}
\end{lemma}

\begin{proof}
    Let $\gamma\in\gam{ec}$. Since each vertex of $G$ must meet at least one edge of $\gamma$, the handshaking lemma implies that
    \begin{equation*}
        |\gamma| \ge \left\lceil\frac{n}{2}\right\rceil.
    \end{equation*}
    From this, it follows that the constant density $\rho_0\equiv\left\lceil\frac{n}{2}\right\rceil^{-1}$ is admissible, and the bound follows from the observation that $\Mod_2(\gam{ec})\le\mathcal{E}_2(\rho_0).$
\end{proof}

\begin{example}[Cycle Graph]\label{ex:cycle-ec}
Let $G = C_n$ be the unweighted cycle graph with $|V| = n$ and $|E| = n$. Again, symmetry suggests that the extremal density is constant. Lemma~\ref{lem:const-upper-bound-ec} shows that
\begin{equation*}
    \Mod_2(\gam{ec}) \le n\left\lceil\frac{n}{2}\right\rceil^{-2}
    =
    \begin{cases}
        \frac{4}{n} & \text{ if $n$ is even},\\
        \frac{4n}{(n+1)^2} & \text{ if $n$ is odd}.
    \end{cases}
\end{equation*}

    By making the argument by symmetry more precise, using the Symmetry Rule of~\cite[Section~5.3]{albin2017modulus}, it is possible to show that $\rho_0$ from Lemma~\ref{lem:const-upper-bound-ec} is extremal in both cases. For the even cycle, this can also be established by the lower bound found in Example~\ref{ex:cycle-star}.
\end{example}

     \begin{figure}
        \centering
        \includegraphics[width=0.2\textwidth]{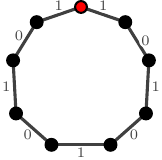}
        \caption{A minimal edge cover for an odd cycle $C$.}
        \label{fig:odd}
    \end{figure}

\begin{example}[Complete Graph]\label{ex:complete-ec}
    Let $G = K_n$ be the unweighted complete graph with $|V| = n$ and $|E| = \frac{n(n-1)}{2}$. Applying Lemma~\ref{lem:const-upper-bound-ec} again shows that
    \begin{equation*}
        \Mod_2(\gam{ec}) \le \frac{n(n-1)}{2}\left\lceil\frac{n}{2}\right\rceil^{-2}
            =
    \begin{cases}
       \frac{2(n-1)}{n}  & \text{ if $n$ is even},\\
       \frac{2n(n-1)}{(n+1)^2} & \text{ if $n$ is odd}.
    \end{cases}
    \end{equation*}
     As in the previous example, the Symmetry Rule can be used to show that these bounds are sharp.
\end{example}

Numerical computation of the edge cover modulus, $\Mod_{p,\sigma}(\gam{ec})$, is challenging because each edge cover induces an inequality constraint in the optimization problem~\eqref{eq:optimization_prob}. On general graphs, this leads to an exponentially large number of constraints. It is possible to show that the family $\gam{ec}$ is equivalent (in the sense of modulus) to the family of minimal edge covers. While this removes many redundant constraints, the number of minimal edge covers of a graph is still generally large.  For example, when $n$ is even, the complete graph, $K_n$, has at least $(n-1)!!$ minimal edge covers. Even the task of enumerating all edge covers of a large graph quickly becomes computationally infeasible. One way to circumvent this combinatorial complexity is to introduce a relaxed problem. This can be accomplished using fractional edge covers.

\subsection{The structure of fractional edge covers}

Since the family $\gam{fec}$ is uncountably infinite, the modulus problem~\eqref{eq:optimization_prob} on this family cannot be immediately viewed as a standard convex optimization problem. However, it is possible to find an equivalent finite family (in the sense of Section~\ref{sec:defs}), $\bfec$. This comes from the observations that $\gam{fec}$ is convex and \emph{recessive}, in the sense that $\gam{fec}=\dom(\gam{fec})$, which implies that $\gam{fec}\modeq\ext(\gam{fec})$. The extreme points of $\gam{fec}$ have a relatively simple structure. 

The following definition, lemmas, and proofs where inspired by the work done in~\cite{scheinerman2011fractional} with fractional perfect matchings. 

\begin{definition}\label{def:bfec}
A vector $\gamma\in\mathbb{R}^{E}_{\ge 0}$ is called a \emph{basic fractional edge cover} if
\begin{itemize}
    \item $\gamma$ is a fractional edge cover taking only values in $\{0,1/2,1\}$,
    \item the support, $\supp\gamma$, is a vertex-disjoint union of odd cycles and substars, and
    \item $\gamma(e)=1/2$ if and only if $e$ belongs to an odd cycle in $\supp\gamma$.
\end{itemize}
The family of all basic fractional edge covers is denoted $\bfec$. The remainder of this section is devoted to showing that the extreme points of $\gam{fec}$ are basic fractional edge covers.
\end{definition}

\begin{lemma}\label{lem:no-even-cycles}
    If $\gamma\in\ext(\gam{fec})$ then $\supp\gamma$ contains no even cycles.
\end{lemma}

\begin{proof}
    Suppose, to the contrary, that $C$ is an even cycle in $\supp\gamma$. Define $\tau\in\mathbb{R}^E$ to be a function that assigns $1$ and $-1$ alternately to the edges of $C$ (starting from an arbitrary edge), and that assigns 0 to all other edges of $G$. Note that, for any number $\alpha$ and any vertex $v$, 
    \begin{equation*}
    (\gamma+\alpha\tau)(\delta(v)) = \gamma(\delta(v))+\alpha\tau(\delta(v)) = \gamma(\delta(v)) \ge 1.
    \end{equation*}
    Thus, for sufficiently small $|\alpha|$, $\gamma + \alpha \tau$ is non-negative and, therefore, a fractional edge cover. This implies that the extreme point $\gamma$ lies in an open line segment in $\gam{fec}$, which is a contradiction.
\end{proof}

\begin{lemma}\label{lem:component-substars}
    Let $\gamma\in\ext(\gam{fec})$. Then, any connected component of $\supp\gamma$ that contains a pendant edge is a substar.
\end{lemma}

\begin{proof}
    Consider a connected component of $H:=\supp\gamma$ containing a pendant edge. We shall show that this component contains no path of length greater than 2, which implies that the component is a substar. Let $e = (u,v)$ be the pendant edge, with $\deg_H(u) = 1$. If $\deg_H(v) = 1$, then the component consists solely of the edge $e$ and we are done. Now, suppose $\deg_H(v) > 1$ and that there is a path of length greater than 2. Starting at node $u$, trace a maximum-length path in $H$. This path, which is assumed to have length at least 3, must either end at another pendant edge in $H$ or cannot be extended further without creating a cycle. We consider these two possibilities in turn.
    
    \begin{enumerate}
        \item[(a)] Assume the path ends in another pendant edge, denoted $\{x,y\}$, with $\deg_H(y)=1$. Since the path has length at least $3$, $x\ne v$.  Let $\tau\in\mathbb{R}^E$ be a function that assigns 0 to the two pendant edges, assigns 1 and $-1$ alternately to the other edges in the path, and assigns 0 to all other edges in $E$ (see Figure~\ref{fig:pendant_edge_proof}). Again, if we show that for sufficiently small $|\alpha|$, $\gamma+\alpha\tau\in\gam{fec}$, we arrive at a contradiction. There are two types of vertices to consider in this case. Consider the vertex $v$. The path in $H$ connecting $u$ to $y$ passes from $u$ to $v$ and then on to a third vertex $w$. (It is possible that $w=x$.) Since $\gamma\in\gam{fec}$, we know that $\gamma(\{u,v\}) \geq 1$ and since $\{v,w\}\in\supp\gamma$, we know that $\gamma(\{v,w\})>0$. So,
        \begin{equation*}
        (\gamma+\alpha\tau)(\delta(v)) = \gamma(\delta(v))+\alpha\tau(\delta(v)) 
        \ge \gamma(\{u,v\}) + \gamma(\{v,w\}) - \alpha
        > 1 - \alpha,
        \end{equation*}
        which is greater than or equal to 1 for sufficiently small $|\alpha|$. The vertex $x$ is similar. For all other vertices, $z$, $\tau(\delta(z))=0$, arriving at the contradiction.

        \begin{figure}
        \centering
            \includegraphics[width=0.5\textwidth]{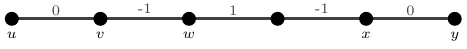}
        \caption{Values of $\tau$ for pendant edges connected by a path.}
        \label{fig:pendant_edge_proof}
        \end{figure}
         
         \item[(b)] Assume instead that the path ends in a cycle. That is, one can trace a path in $H$ starting from $u$, passing through $v$, and continuing until the path eventually circles back to repeat a vertex $w$. (It is possible that $w=v$.) From Lemma~\ref{lem:no-even-cycles}, the cycle through $w$ must be odd. Let $\tau\in\mathbb{R}^E$ be a function, shown in Figure~\ref{fig:oddcycle_pendantedge}, that assigns 0 to the pendant edge, assigns $\pm 1/2$ alternately on the odd cycle with the two edges through $w$ sharing the same value, and assigns 0 to all other edge in $E$. If $w\ne v$, then $\tau$ should alternately assign $\pm 1$ to the edges connecting $v$ and $w$ in such a way that $\tau(\delta(w))=0$. (The value of $\tau$ is set to zero on all other edges.)

         It can be seen that this once again leads to a contradiction of the assumption that $\gamma$ is an extreme point. The function $\tau$ sums to zero on all stars other than $\delta(v)$ and $\gamma(\delta(v))>1$, which makes $\gamma + \alpha\tau\in\gam{fec}$ for sufficiently small $|\alpha|$.         
          
    \begin{figure}
        \centering
        \includegraphics[width=0.4\textwidth]{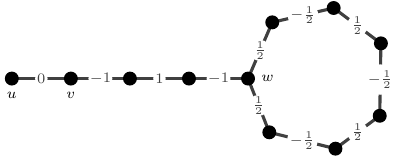}
        \caption{Values of $\tau$ for a pendant edge connected to a cycle.}
        \label{fig:oddcycle_pendantedge}
    \end{figure}
    \end{enumerate}
\end{proof}

\begin{lemma}\label{lem:component-cycles}
    Let $\gamma\in\ext(\gam{fec})$. Then, any connected component of $\supp\gamma$ that does not contain a pendant edge is an odd cycle.
\end{lemma}

\begin{proof}    
    
    \begin{figure}
        \centering
        \includegraphics[width=0.3\textwidth]{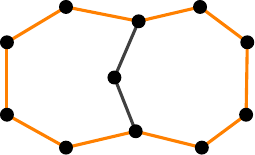}
        \caption{A path that returns to an odd cycle $C$ and creates an even cycle, i.e. the ``snowman'' graph. Orange edges make up the even cycles.}
        \label{fig:snowman}
    \end{figure}    

    Let $H'$ be a connected component of $H := \supp\gamma$ that does not contain a pendant edge. The $H'$ must contain at least one cycle, $C$. By Lemma~\ref{lem:no-even-cycles}, any such cycle must be odd. As before, we proceed by contradiction. Suppose that $H'\setminus C\ne\emptyset$. Then $H'$ must contain a vertex $v$ such that $\deg_H(v) \geq 3$. That is, there must be an edge of $H$ that is incident to $v$ but does not lie in the cycle $C$. Consider extending this edge to a maximum-length path in $H$. Since the component is assumed to contain no pendant edges, this path must eventually create a cycle.  There are a few cases to consider.

    One possibility is that the path leaving $v$ returns to a different vertex $w$ in $C$, producing a ``snowman'' as shown in Figure~\ref{fig:snowman}. The vertices $v$ and $w$ are connected by two paths in $C$, and since $C$ is an odd cyle, one of these paths has an even number of edges while the other has an odd number. Removing the path with even length results in an even cycle, which contradicts Lemma~\ref{lem:no-even-cycles}.
    
    So, the path leaving $C$ from $v$ does not return to any other vertex of $C$ (and does not end in a pendant edge). Thus, $H$ must contain a graph $K$ with 2 (necessarily odd) cycles connected by a path (possibly of length 0) as in Figure~\ref{fig:oddcyclecases}. Let $\tau\in\mathbb{R}^E$ be a function that assigns 0 to edges not in $K$, $\pm 1$ alternately on the path connecting the two cycles of $K$, and $\pm 1/2$ alternately around the cycles of $K$ so that $\tau$ sums to zero on all stars of $G$. (See Figure \ref{fig:odd_cycle_barbell} and Figure \ref{fig:bowtie} for examples.) This again yields a contradiction, since $\gamma+\alpha\tau\in\gam{fec}$ for sufficiently small $|\alpha|$.
    
     \begin{figure}
     \centering 
     
     \begin{subfigure}[b]{0.4\textwidth}
     \begin{center}
             \begin{tikzpicture}
                \Vertex[x=0.3, y=-1.187, size=.2,color=black]{1}
                \Vertex[x=1.05, y=-0.63, size=.2,color=black]{2}
                \Vertex[x=1.07, y=0.596, size=.2,color=black]{3}
                \Vertex[x=0.27, y=1.195, size=.2,color=black]{4}
                \Vertex[x=-0.77, y=0.952, size=.2,color=black]{5}
                \Vertex[x=-1.18, y=0, size=.2,color=black]{6}
                \Vertex[x=-0.82, y=-0.91, size=.2,color=black]{7}
                \Vertex[x=-2.18, y=0, size=.2,color=black]{A}
                \Vertex[x=-3.18, y=0, size=.2,color=black]{B}
                \Vertex[x=-4.18, y=0, size=.2,color=black]{C}
                \Vertex[x=-5.18, y=0, size=.2,color=black]{D}
                \Vertex[x=-5.77, y=0.952, size=.2,color=black]{8}
                \Vertex[x=-6.73, y=1.195, size=.2,color=black]{9}
                \Vertex[x=-7.57, y=0.596, size=.2,color=black]{10}
                \Vertex[x=-7.55, y=-0.63, size=.2,color=black]{11}
                \Vertex[x=-6.7, y=-1.187, size=.2,color=black]{12}
                \Vertex[x=-5.77, y=-0.91, size=.2,color=black]{13}
                \Edge[label=$\frac{1}{2}$](1)(2)
                \Edge[label=$-\frac{1}{2}$](2)(3)
                \Edge[label=$\frac{1}{2}$](3)(4)
                \Edge[label=$-\frac{1}{2}$](4)(5)
                \Edge[label=$\frac{1}{2}$](5)(6)
                \Edge[label=$\frac{1}{2}$](6)(7)
                \Edge[label=$-\frac{1}{2}$](7)(1)
                \Edge[label=$-1$](6)(A)
                \Edge[label=1](B)(A)
                \Edge[label=$-1$](B)(C)
                \Edge[label=1](C)(D)
                \Edge[label=$-\frac{1}{2}$](D)(8)
                \Edge[label=$-\frac{1}{2}$](D)(13)
                \Edge[label=$\frac{1}{2}$](9)(8)
                \Edge[label=$-\frac{1}{2}$](10)(9)
                \Edge[label=$\frac{1}{2}$](11)(10)
                \Edge[label=$-\frac{1}{2}$](12)(11)
                \Edge[label=$\frac{1}{2}$](12)(13)
            \end{tikzpicture}
        \end{center}
        \caption{}
        \label{fig:odd_cycle_barbell}
     \end{subfigure}
     \hfill
    \begin{subfigure}[b]{0.4\textwidth}
         \centering
            \begin{tikzpicture}
               \Vertex[x=0.3, y=-1.187, size=.2,color=black]{1}
                \Vertex[x=1.05, y=-0.63, size=.2,color=black]{2}
                \Vertex[x=1.07, y=0.596, size=.2,color=black]{3}
                \Vertex[x=0.27, y=1.195, size=.2,color=black]{4}
                \Vertex[x=-0.77, y=0.952, size=.2,color=black]{5}
                \Vertex[x=-1.18, y=0, size=.2,color=black]{6}
                \Vertex[x=-0.82, y=-0.91, size=.2,color=black]{7}
                \Vertex[x=-1.77, y=0.952, size=.2,color=black]{8}
                \Vertex[x=-2.73, y=1.195, size=.2,color=black]{9}
                \Vertex[x=-3.57, y=0.596, size=.2,color=black]{10}
                \Vertex[x=-3.55, y=-0.63, size=.2,color=black]{11}
                \Vertex[x=-2.7, y=-1.187, size=.2,color=black]{12}
                \Vertex[x=-1.77, y=-0.91, size=.2,color=black]{13}
                \Edge[label=$\frac{1}{2}$](1)(2)
                \Edge[label=$-\frac{1}{2}$](2)(3)
                \Edge[label=$\frac{1}{2}$](3)(4)
                \Edge[label=$-\frac{1}{2}$](4)(5)
                \Edge[label=$\frac{1}{2}$](5)(6)
                \Edge[label=$\frac{1}{2}$](6)(7)
                \Edge[label=$-\frac{1}{2}$](7)(1)
                \Edge[label=$-\frac{1}{2}$](6)(8)
                \Edge[label=$-\frac{1}{2}$](6)(13)
                \Edge[label=$\frac{1}{2}$](9)(8)
                \Edge[label=$-\frac{1}{2}$](10)(9)
                \Edge[label=$\frac{1}{2}$](11)(10)
                \Edge[label=$-\frac{1}{2}$](12)(11)
                \Edge[label=$\frac{1}{2}$](12)(13)
            \end{tikzpicture}
         \caption{}
         \label{fig:bowtie}
     \end{subfigure}
        \caption{Values of $\tau$ for two connected cycles.}
        \label{fig:oddcyclecases}
    \end{figure}    
\end{proof}

\begin{lemma}\label{lem:extreme-bfec}
    Every extreme point in $\ext(\gam{fec})$ is a basic fractional edge cover.
\end{lemma}

\begin{proof}
Lemmas~\ref{lem:component-substars} and~\ref{lem:component-cycles} show that each connected component of $H:=\supp\gamma$ is either a substar or an odd cycle in $G$. To complete the proof, we must show that $\gamma$ only takes values in $\{0,1/2,1\}$, with the value $1/2$ occurring exactly on the odd cycles of $H$.

First, consider a substar component, $S$, of $H$. All edges in this case are pendant edges and, therefore, $\gamma$ must be at least 1 on each edge. On the other hand, suppose $\gamma(e')>1$ for some $e'\in S$, and define
\begin{equation*}
    \tilde{\gamma}(e) =
    \begin{cases}
        1 & \text{if }e=e',\\
        \gamma(e) &\text{otherwise}.
    \end{cases}
\end{equation*}
Then $\tilde{\gamma}\in\gam{fec}$ and $\tilde{\gamma}\preceq\gamma$. Since $\gam{fec}$ is recessive, this implies that $\gamma$ lies on the relative interior of a ray in $\gam{fec}$ emanating from $\tilde{\gamma}$ and, therefore, that $\gamma$ cannot be an extreme point.

Next, consider a component, $C$, of $H$ comprising an odd cycle. We wish to show that $\gamma(e)=1/2$ on all edges of $C$. We begin by observing that $\gamma(\delta(v))=1$ for every vertex $v$ in the cycle. Suppose to the contrary that $\gamma(\delta(v))>1$ for some vertex $v$, and consider the vector $\tau\in\mathbb{R}^E$, supported on $C$, alternately taking the values $\pm 1$ around the cycle in such a way that $+1$ is assigned to both edges incident on $v$. Then, $\gamma+\alpha\tau\in\gam{fec}$ for sufficiently small $|\alpha|$, contradicting the extremality of $\gamma$.

Next, we argue that if $\gamma(\delta(v))=1$ for all $v\in C$ then $\gamma(e)=1/2$ for all edges in $C$. To see this, define
\begin{equation*}
    \tilde{\gamma}(e) = 
    \begin{cases}
    1/2 &\text{if }e\in C,\\
    \gamma(e)&\text{otherwise}.        
    \end{cases}
\end{equation*}
Then, $\tilde{\gamma}\in\gam{fec}$ and $\tilde{\gamma}(\delta(v))=1$ for every vertex $v\in C$. Define $\tau:=\gamma-\tilde{\gamma}$. Then $\supp\tau\subseteq C$ and $\tau(\delta(v))=0$ for every $v\in C$. Choose an adjacent pair of edges $e_1,e_2\in C$. Then $\tau(e_2)=-\tau(e_1)$. Continuing around the cycle we find that the next edge, $e_3$, must satisfy $\tau(e_3)=-\tau(e_2)=\tau(e_1)$ and so on. Since the cycle is odd, the last edge we cross, $e_r$, that completes the cycle must have $\tau(e_r)=\tau(e_1)$. But, since $\tau$ must sum to zero on the star including these two edges, $\tau(e_r)=\tau(e_1)=0$, implying that $\tau=0$. Thus, $\gamma=\tilde{\gamma}$, implying that $\gamma(e)=1/2$ on all edges $e\in C$.
\end{proof}

\begin{theorem}\label{thm:fec-bfec-equiv}
The families $\gam{fec}$ and $\bfec$ are equivalent; $\gam{fec}\modeq\bfec$.
\end{theorem}

\begin{proof}
Lemma~\ref{lem:extreme-bfec} implies that $\ext(\gam{fec})\subseteq\bfec\subseteq\gam{fec}$, which shows that
$\adm(\gam{fec})\subseteq\adm(\bfec)\subseteq\adm(\ext(\gam{fec}))=\adm(\gam{fec})$.
\end{proof}

Figures~\ref{fig:W5fec} and~\ref{fig:W6fec} demonstrate the implication of Theorem~\ref{thm:fec-bfec-equiv}. These figures show all extreme points (up to rotation and reflection) of $\gam{fec}$ for the wheel graphs $W_5$ and $W_6$ respectively. Note that (as guaranteed by the theorem) all are basic fractional edge covers. Any density that is admissible for one of these sets is admissible for all fractional edge covers on the corresponding graph.

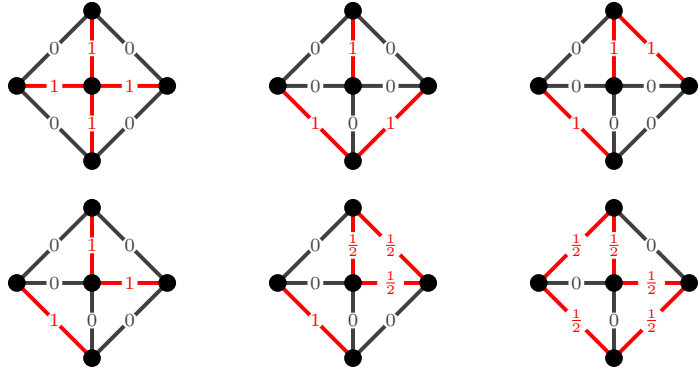
\begin{figure}

    \begin{center}
        \begin{tikzpicture}
        \Vertex[x=0,y=0,color=black,size=0.2]{0}
        \Vertex[x=0,y=1,color=black,size=0.2]{1}
        \Vertex[x=1,y=0,color=black,size=0.2]{2}
        \Vertex[x=0,y=-1,color=black,size=0.2]{3}
        \Vertex[x=-1,y=0,color=black,size=0.2]{4}
        \Edge[label=1,color=red](0)(1)
        \Edge[label=1,color=red](0)(2)
        \Edge[label=1,color=red](0)(3)
        \Edge[label=1,color=red](0)(4)
        \Edge[label=0](2)(1)
        \Edge[label=0](3)(2)
        \Edge[label=0](4)(1)
        \Edge[label=0](3)(4)
        \end{tikzpicture}
        \hspace{1cm}
        \begin{tikzpicture}
        \Vertex[x=0,y=0,color=black,size=0.2]{0}
        \Vertex[x=0,y=1,color=black,size=0.2]{1}
        \Vertex[x=1,y=0,color=black,size=0.2]{2}
        \Vertex[x=0,y=-1,color=black,size=0.2]{3}
        \Vertex[x=-1,y=0,color=black,size=0.2]{4}
        \Edge[label=1,color=red](0)(1)
        \Edge[label=0](0)(2)
        \Edge[label=0](0)(3)
        \Edge[label=0](0)(4)
        \Edge[label=0](2)(1)
        \Edge[label=1,color=red](3)(2)
        \Edge[label=0](4)(1)
        \Edge[label=1,color=red](3)(4)
        \end{tikzpicture}
        \hspace{1cm}
        \begin{tikzpicture}
        \Vertex[x=0,y=0,color=black,size=0.2]{0}
        \Vertex[x=0,y=1,color=black,size=0.2]{1}
        \Vertex[x=1,y=0,color=black,size=0.2]{2}
        \Vertex[x=0,y=-1,color=black,size=0.2]{3}
        \Vertex[x=-1,y=0,color=black,size=0.2]{4}
        \Edge[label=1,color=red](0)(1)
        \Edge[label=0](0)(2)
        \Edge[label=0](0)(3)
        \Edge[label=0](0)(4)
        \Edge[label=1,color=red](2)(1)
        \Edge[label=0](3)(2)
        \Edge[label=0](4)(1)
        \Edge[label=1,color=red](3)(4)
        \end{tikzpicture}
    \end{center}
    \begin{center}
        \begin{tikzpicture}
        \Vertex[x=0,y=0,color=black,size=0.2]{0}
        \Vertex[x=0,y=1,color=black,size=0.2]{1}
        \Vertex[x=1,y=0,color=black,size=0.2]{2}
        \Vertex[x=0,y=-1,color=black,size=0.2]{3}
        \Vertex[x=-1,y=0,color=black,size=0.2]{4}
        \Edge[label=1, color=red](0)(1)
        \Edge[label=1,color=red](0)(2)
        \Edge[label=0](0)(3)
        \Edge[label=0](0)(4)
        \Edge[label=0](2)(1)
        \Edge[label=0](3)(2)
        \Edge[label=0](4)(1)
        \Edge[label=1, color=red](3)(4)
        \end{tikzpicture}
        \hspace{1cm}
        \begin{tikzpicture}
        \Vertex[x=0,y=0,color=black,size=0.2]{0}
        \Vertex[x=0,y=1,color=black,size=0.2]{1}
        \Vertex[x=1,y=0,color=black,size=0.2]{2}
        \Vertex[x=0,y=-1,color=black,size=0.2]{3}
        \Vertex[x=-1,y=0,color=black,size=0.2]{4}
        \Edge[label=$\frac{1}{2}$,color=red](0)(1)
        \Edge[label=$\frac{1}{2}$,color=red](0)(2)
        \Edge[label=0](0)(3)
        \Edge[label=0](0)(4)
        \Edge[label=$\frac{1}{2}$,color=red](2)(1)
        \Edge[label=0](3)(2)
        \Edge[label=0](4)(1)
        \Edge[label=1,color=red](3)(4)
        \end{tikzpicture}
        \hspace{1cm}
        \begin{tikzpicture}
        \Vertex[x=0,y=0,color=black,size=0.2]{0}
        \Vertex[x=0,y=1,color=black,size=0.2]{1}
        \Vertex[x=1,y=0,color=black,size=0.2]{2}
        \Vertex[x=0,y=-1,color=black,size=0.2]{3}
        \Vertex[x=-1,y=0,color=black,size=0.2]{4}
        \Edge[label=$\frac{1}{2}$,color=red](0)(1)
        \Edge[label=$\frac{1}{2}$,color=red](0)(2)
        \Edge[label=0](0)(3)
        \Edge[label=0](0)(4)
        \Edge[label=0](2)(1)
        \Edge[label=$\frac{1}{2}$,color=red](3)(2)
        \Edge[label=$\frac{1}{2}$,color=red](4)(1)
        \Edge[label=$\frac{1}{2}$,color=red](3)(4)
        \end{tikzpicture}
    \end{center}
        
\caption{Extreme points of $\gam{fec}$ (up to rotation and reflection) for $G = W_5$. There are 25 extreme points in total.}
\label{fig:W5fec}

\end{figure}

\begin{figure}
        \begin{center}
        \includegraphics[width=0.8in]{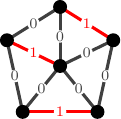}
        \hspace{0.6cm}
        \includegraphics[width=0.8in]{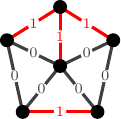}
        \hspace{0.6cm}
        \includegraphics[width=0.8in]{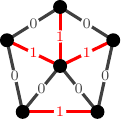}
        \hspace{0.6cm}
        \includegraphics[width=0.8in]{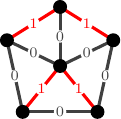}
        \end{center}
        \begin{center}
        \includegraphics[width=0.8in]{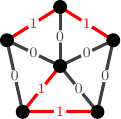}
        \hspace{0.6cm}
        \includegraphics[width=0.8in]{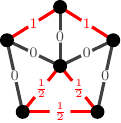}
        \hspace{0.6cm}
            \includegraphics[width=0.8in]{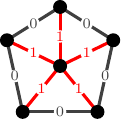}
        \end{center}
    \caption{Extreme points of $\gam{fec}$ (up to rotation and reflection) for $G = W_6$. There are 36 extreme points in total.}
    \label{fig:W6fec}
\end{figure}

\begin{corollary}\label{cor:bipartite-equiv}
If $G$ is a bipartite graph, then $\gam{fec}\modeq\gam{ec}$.
\end{corollary}

\begin{proof}
    Let $\gamma\in\bfec$.  Since a bipartite graph has no odd cycles, $\gamma$ must take values in $\{0,1\}$. Moreover, $\gamma(\delta(v))\ge 1$ for every $v\in V$ and, therefore, $\gamma$ is (the incidence vector of) an edge cover.
\end{proof}

\subsection{Computing $\Mod_{p}(\gam{fec})$}

We now revisit the graphs we studied at the beginning of this section and calculate the unweighted 2-modulus of fractional edge covers.  By Theorem \ref{thm:fec-bfec-equiv}, we have that $\gam{fec} \modeq \bfec$. Therefore, calculating $\Mod_2(\gam{fec})$ is equivalent to calculating $\Mod_2(\bfec)$.

\begin{example}[Star Graph]\label{ex:star-fec}
Let $G = S_n$ be the unweighted star graph. Note that $S_n$ is a bipartite graph, so by Corollary \ref{cor:bipartite-equiv}, $\gam{ec} \modeq \gam{fec}$ and 
\[
\Mod_2(\gam{fec}) = \Mod_2(\gam{ec}) =  \frac{1}{n-1}.
\]
\end{example}

\begin{example}[Cycle Graph]\label{ex:cycle-fec}
Let $G = C_n$ be the unweighted cycle graph. To calculate the fractional edge cover modulus, we again consider the cases when $n$ is even or odd. 
\begin{itemize}
    \item[(a)] If $n$ is even, then the graph is bipartite. Again, by Corollary \ref{cor:bipartite-equiv}, $\gam{ec} \modeq \gam{fec}$ and 
    \[
    \Mod_2(\gam{fec}) = \Mod_2(\gam{ec}) = \frac{4}{n}.
    \]

    \item[(b)] If $n$ is odd, then $\bfec$ is the union of all the minimal edge covers and the constant vector $\gamma \equiv 1/2$.
    It is straightforward to verify that the constant density $\rho_0 \equiv \frac{2}{n}$ is admissible for this family. Thus,
    \[
    \Mod_2(\Gamma) \le \cE_2(\rho_0) = \sum_{e \in E} \rho_0(e)^2 = n \cdot \frac{4}{n^2} = \frac{4}{n}.
    \]
    The corresponding lower bound is established in Example~\ref{ex:cycle-star} of Section~\ref{sec:star_mod}.
\end{itemize}
\end{example}

\begin{example}[Complete Graph]\label{ex:complete-fec}
    Let $G = K_n$ be the unweighted complete graph with $|V| = n$ and $|E| = \frac{n(n-1)}{2}$. In this example, the number of basic fractional edge covers grows quickly with $n$, and is difficult to visualize. Nevertheless, we can show that $\rho_0 = \frac{2}{n}$ is the extremal density and that the modulus is 
    \[
    \Mod_2(\gam{fec}) = \frac{2(n-1)}{n}.
    \]
    This is proved in Example~\ref{ex:complete-star} in Section~\ref{sec:star_mod}.
\end{example}

\subsection{Bounds on the modulus of edge covers}

One reason that fractional edge covers are interesting in the context of modulus comes from the following theorem.

\begin{theorem}\label{thm:bounds}
For all $p\in[1,\infty)$
\[ 
\left( \frac{3}{4} \right)^p \Mod_{p,\sigma}(\gam{fec})
\leq \Mod_{p,\sigma}(\gam{ec})
\leq \Mod_{p,\sigma}(\gam{fec}).
\]
If $G$ is bipartite then $\Mod_{p,\sigma}(\gam{ec})
= \Mod_{p,\sigma}(\gam{fec})$.
\end{theorem}

The second inequality in these bounds follows from the inclusion $\gam{ec} \subseteq \gam{fec}$. The other bound follows from a more careful look at cycles.

First, we establish an upper bound on the number of edges in a minimal edge cover for a cycle.

\begin{lemma}\label{lem:cycle-cover-upper-bound}
    Let $C$ be a cycle with length $|C|=k$ and let $\gamma$ be a minimal edge cover of $C$. Then $|\gamma|\le\frac{2k}{3}$.
\end{lemma}
\begin{proof}
    Each vertex of $C$ has degree $1$ or $2$ in $\gamma$. Let $k_1$ be the number of vertices with degree $1$ and $k_2=k-k_1$ be the vertices with degree $2$. Since $\gamma$ is minimal, each degree-2 vertex is connected in $\gamma$ to two degree-1 neighbors. (An edge connecting two degree-2 vertices could be removed from $\gamma$ leaving an edge cover and violating the minimality of $\gamma$.) This implies that $k_2\le\frac{k}{3}$. The handshaking lemma then implies that
    \begin{equation*}
        |\gamma| = \frac{k_1+2k_2}{2} = \frac{k+k_2}{2} \le \frac{2k}{3}.
    \end{equation*}
\end{proof}

The key estimate is the following.

\begin{lemma}\label{lem:rho-len-odd-cycle}
    Let $C\subset E$ be a cycle, let $\gam{ec}(C)$ be the set of edge covers for $C$, and let $\rho\in\mathbb{R}^C_{\ge 0}$. Then
    \begin{equation*}
        \frac{1}{2}\sum_{e\in C}\rho(e)
        \ge \frac{3}{4}\min_{\gamma\in\gam{ec}(C)}\ell_\rho(\gamma)
    \end{equation*}
\end{lemma}

\begin{proof}
    Let $C$ be a cycle with length $|C|=k$. Since $\rho$ is non-negative, the minimum value on the right must be attained by a minimal edge cover. Let $\gamma' = \argmin_{\gamma \in \gam{ec}(C)} \ell_{\rho}(\gamma)$, and assume $\gamma'$ is minimal. By considering all rotations of the cycle, we obtain a set of $k$ rotations of $\gamma'$, which we may enumerate $\{\gamma_1',\gamma_2',\ldots,\gamma_k'\}$. Moreover, each edge of $C$ lies in exactly $|\gamma'|$ of these rotations, so, by Lemma~\ref{lem:cycle-cover-upper-bound},
    \begin{equation*}
        \frac{3}{4k}\sum_{i=1}^k\gamma_i = \frac{3}{4k}|\gamma'| \mathbf{1} \le \frac{3}{4k}\frac{2k}{3} \mathbf{1} = \frac{1}{2}\mathbf{1}.
    \end{equation*}
    From this, it follows that
    \begin{equation*}
        \frac{1}{2}\sum_{e\in C}\rho(e)
        =\frac{1}{2}\rho^T\mathbf{1}
        \ge \frac{3}{4k}\sum_{i=1}^k\ell_\rho(\gamma_i)
        \ge \frac{3}{4}\min_{\gamma\in\gam{ec}(C)}\ell_\rho(\gamma).
    \end{equation*}  
\end{proof}

\begin{lemma}\label{lem:almost-admissible}
Let $\rho \in \adm(\gam{ec})$, then $\displaystyle \frac{4}{3}\rho \in \adm(\bfec)$.
\end{lemma}

\begin{proof}
Let $\rho \in \adm(\gam{ec})$ and let $\gamma \in \Gamma_{\overline{fec}}$. From Definition~\ref{def:bfec}, we know that the support of $\gamma$ is a vertex-disjoint union of substars, $\mathscr{S} = \{S_1, \ldots, S_t\}$, and odd cycles, $\mathscr{C} =\{C_1, \ldots, C_r\}$. Moreover, $\gamma$ takes the value $1$ on each substar edge and $1/2$ on each cycle edge. From $\gamma$, we construct an edge cover, $\tilde{\gamma}\in\gam{ec}$ as follows.

For edges $e\notin\supp\gamma$, we set $\tilde{\gamma}(e)=0$. Similarly, we set $\tilde{\gamma}(e)=\gamma(e)=1$ for each edge $e\in\bigcup_{i=1}^tS_i$. The remaining edges lie on the disjoint union of the odd cycles in $\mathscr{C}$. On each of these odd cycles, we choose $\tilde{\gamma}$ to be the (incidence vector of the) minimal edge cover that has the smallest $\rho$-length. By construction, $\tilde{\gamma}$ is an edge cover for $G$. By admissibility, then, it follows that
\begin{equation}\label{eq:gamma-tilde-adm}
    1 \le \ell_\rho(\tilde{\gamma})
    = \sum_{C\in\mathscr{C}}\sum_{e\in C}\tilde{\gamma}(e)\rho(e) + 
    \sum_{S\in\mathscr{S}}\sum_{e\in S}\rho(e).
\end{equation}

Now consider the $\rho$-length of the basic fractional edge cover $\gamma$,
\begin{equation}\label{eq:rho-len-gamma}
    \ell_\rho(\gamma) =
    \sum_{C\in\mathscr{C}}\sum_{e\in C}\frac{1}{2}\rho(e) + 
    \sum_{S\in\mathscr{S}}\sum_{e\in S}\rho(e).
\end{equation}
By the construction of $\tilde{\gamma}$, Lemma~\ref{lem:rho-len-odd-cycle} implies that for every $C\in\mathscr{C}$,
\begin{equation}\label{eq:cycle-compare}
    \frac{1}{2}\sum_{e\in C}\rho(e)
    \ge \frac{3}{4}\sum_{e\in C}\tilde{\gamma}(e)\rho(e).
\end{equation}
Combining~\eqref{eq:gamma-tilde-adm}--\eqref{eq:cycle-compare} shows that
\begin{equation*}
    \ell_\rho(\gamma) \ge \frac{3}{4}.
\end{equation*}
Since $\gamma\in\bfec$ was arbitrary, it follows that $\frac{4}{3}\rho\in\adm(\bfec)$.

\end{proof}

\begin{proof}[Proof of Theorem~\ref{thm:bounds}]
If $G$ is bipartite, then the result follows from Corollary~\ref{cor:bipartite-equiv}. Otherwise, as stated earlier, the second inequality follows from the inclusion $\gam{ec}\subseteq\bfec$. The first inequality is a consequence of Lemma~\ref{lem:almost-admissible}. To see this, let $\rho^*$ be an extremal density for $\Mod_{p,\sigma}(\gam{ec})$. By the lemma, $\rho = \frac{4}{3}\rho^*\in\adm(\bfec)=\adm(\gam{fec})$. Thus, we have
\begin{equation*}
    \Mod_{p,\sigma}(\gam{fec})
    \le \mathcal{E}_{p,\sigma}(\rho)
    = \left(\frac{4}{3}\right)^p\mathcal{E}_{p,\sigma}(\rho^*) = \left(\frac{4}{3}\right)^p\Mod_{p,\sigma}(\gam{ec}).
\end{equation*}
\end{proof}

\begin{remark}
    A similar theorem is true for the case $p=\infty$. Using the same proof technique, one finds that
\[ 
\frac{3}{4} \Mod_{\infty, \sigma}(\gam{fec}) \le \Mod_{\infty, \sigma}(\gam{ec}) \le \Mod_{\infty, \sigma}(\gam{fec}).
\]    
\end{remark}

\section{Fulkerson Duality}\label{sec:fulkerson_duality}

The theory of Fulkerson Duality applied to modulus was developed in \cite{albin2019blocking}. If $\Gamma$ is a finite family, then the admissible set, $\adm(\Gamma)$, has finitely many faces and finitely many extreme points. Since $\adm(\Gamma)$ is a recessive closed convex set, it can be written as the dominant of its extreme points
\[
\adm(\Gamma) = \dom(\ext(\adm(\Gamma))).
\]
We define
\[
\hat{\Gamma} := \ext( \adm(\Gamma)) = \{ \hat{\gamma}_1, \ldots, \hat{\gamma}_r\} \subseteq \bR^E_{\ge 0}
\]
to be the \emph{Fulkerson blocker of $\Gamma$}. These extreme points can be thought of as another family of graph objects with usage matrix $\hat{\cN} \in \bR^{\hat{\Gamma} \times E}_{\ge 0}$. This construction provides a duality among families of objects due to the fact that
\begin{equation*}
\hat{\hat{\Gamma}} = \ext(\dom(\Gamma)) \modeq\Gamma.
\end{equation*}
The relationship between $\Gamma$ and $\hat{\Gamma}$ in terms of modulus is given by the following theorem.

\begin{theorem}[Theorem 4 in \cite{albin2019blocking}]\label{thm:mod_reciprocals}
    Let $G = (V,E)$ be a graph and let $\Gamma$ be a non-trivial finite family of objects on G with Fulkerson blocker $\hat{\Gamma}$. Let the exponent $1 < p < \infty$ be given, with $q := p/(p-1)$ its Hölder conjugate. For any set of weights $\sigma \in \bR_{>0}^E$, define the dual set of weights $\hat{\sigma}$ as $\hat{\sigma}(e) := \sigma(e)^{-\frac{q}{p}}$, for all $e \in E$. Then, 
    \begin{equation}\label{eq:mod_reciprocal}
        \Mod_{p,\sigma}(\Gamma)^{1/p}\Mod_{q,\hat{\sigma}}(\hat{\Gamma})^{1/q} = 1. 
    \end{equation} 

    Moreover, the optimal $\rho^* \in \adm(\Gamma)$ and $\eta^* \in \adm(\hat{\Gamma})$ are unique and are related as follows: 
    \begin{equation}
        \eta^*(e) = \frac{\sigma(e) \rho^*(e)^{p-1}}{\Mod_{p,\sigma}(\Gamma)} \qquad \forall e \in E.
    \end{equation}
    
\end{theorem}

The relationship of the modulus of the families when $p=1$ and $p=\infty$ are described using the following theorem. 

\begin{theorem}[Thereom 5 in \cite{albin2019blocking}]
\label{thm:mod_reciprocals-1-inf}
    Under the assumptions of Theorem \ref{thm:mod_reciprocals}, 
    \[ 
    \Mod_{1,\sigma}(\Gamma)\Mod_{\infty,\sigma^{-1}} (\hat{\Gamma}) = 1,
    \]
    where $\sigma^{-1}(e) = \sigma(e)^{-1}$.    
\end{theorem}

The relationship between dual families of objects has several useful implications. Most important in the present setting is the fact that upper bounds on the modulus of the dual family provide lower bounds for the original (primal) family. In particular, the following lemma shows that the family of stars can be used to provide lower bounds for the modulus of fractional edge covers.

\begin{lemma}\label{lem:Bhat_fec}
The families of stars and fractional edge covers on $G$ are dual in the sense that
\begin{equation*}
    \hat{\Gamma}_{\text{star}}\modeq\gam{fec}.
\end{equation*}
\end{lemma}

\begin{proof}
Consider a vector $\gamma\in\mathbb{R}^E_{\ge 0}$. Then, $\gamma\in\gam{fec}$ if and only if
\[
\sum_{e\in\delta(v)}\gamma(e) \ge 1\quad
\text{ for every $v$ in $V$},
\]
which is equivalent to saying that $\gamma\in\adm(\gam{star})$. By Lemma~\ref{lem:gamma-equiv-dom}, then, 
\begin{equation*}
\hat{\Gamma}_\text{star} = \ext(\adm(\gam{star}))
\modeq \adm(\gam{star}) = \gam{fec}.
\end{equation*}
\end{proof}

\section{Star modulus}\label{sec:star_mod}

Lemma~\ref{lem:Bhat_fec} implies that Theorems~\ref{thm:mod_reciprocals} and~\ref{thm:mod_reciprocals-1-inf} hold with $\Gamma=\gam{star}$ and $\hat{\Gamma}=\gam{fec}$. This means that the modulus and extremal densities for fractional edge covers can be understood through the modulus of stars.

Calculating the star modulus turns out to be computationally simpler than calculating the modulus of fractional edge covers based on the number of constraints in the minimization problem. Specifically, the number of stars in a graph is equal to the number of vertices $|V|$, whereas the family of basic fractional edge covers in a graph is at least as big as the family of minimal edge covers. In this section, we prove simple results for the modulus of stars, as well as studying examples of well-known graphs. 

The following lemma states a basic estimate for the star modulus by restating a result from \cite{albin2017modulus}, along with a lower bound.

\begin{lemma}\label{lem:simple-bound-mod-star}
    Let $G = (V,E,\sigma)$ be a finite graph and let $\Gamma$ be the family of all stars in $G$. Let $\delta(G) := \min_{v \in V} \deg(v)$, then 
    \[
    \frac{\sigma_{\text{min}}}{\delta(G)^p} \le \Mod_{p,\sigma}(\Gamma) \le \frac{\sigma(E)}{\delta(G)^p}
    \]
    where $\sigma_{\text{min}} = \min\limits_{e \in E} \sigma(e)$ and $\sigma(E) = \sum\limits_{e \in E} \sigma(e)$.
\end{lemma}

\begin{proof}
    Define $\rho_0 \equiv \frac{1}{\delta(G)}$, then $\rho_0$ is admissible since for every $v \in V$
    \[
    \ell_{\rho_0}(\delta(v)) = \sum_{e \in \delta(v)} \rho_0(e) = \frac{\deg(v)}{\delta(G)} \ge 1,
    \]
    where the last inequality is true since every star will have at least $\delta(G)$ edges in it. So,
    \[
    \Mod_{p,\sigma}(\Gamma) \le \cE_{p,\sigma}(\rho_0) = \sum_{e \in E} \sigma(e) \rho_0(e)^p = \frac{\sigma(E)}{\delta(G)^p}.
    \]
    Let $v_0 \in V$ be the node with smallest degree ($\deg(v_0) = \delta(G)$) and let $\rho \in \adm(\Gamma)$. Then, 
    \[ 1 \le \ell_{\rho}(\delta(v_0)) = \sum_{e \in \delta(v_0)} \rho(e) \le \max_{e\in E} \rho(e) \sum_{e \in \delta(v_0)} 1 = \max_{e \in E} \rho(e) \delta(G)  \]
    This implies that there exists an edge $e_0$ such that $\rho(e_0) \ge \frac{1}{\delta(G)}$ and 
    \[
    \cE_{p,\sigma} (\rho) = \sum_{e \in E} \sigma(e) \rho(e)^p \ge \sigma(e_0)\rho(e_0)^p = \frac{\sigma(e_0)}{\delta(G)^p} \ge \frac{\sigma_{\text{min}}}{\delta(G)^p}.
    \]
    Taking the infimum over all $\rho \in \adm(\Gamma)$ we get the result 
    \[
    \Mod_{p,\sigma} (\Gamma) \ge \frac{\sigma_{\text{min}}}{\delta(G)^p}.
    \]
\end{proof}

\begin{example}[Star Graph]\label{ex:star-star}
Let $G = S_n$ be the unweighted star graph. The density $\rho_0 \equiv 1$ is admissible for $\gam{star}$. So,
\[ 
\Mod_2(\gam{star}) \le \cE_2(\rho_0) = \sum_{e \in E} \rho_0(e)^2 = (n-1) \cdot  1^2 = n-1.
\]
Duality and Example~\ref{ex:star-fec} imply that $\Mod_2(\gam{star})=\Mod_2(\gam{fec})^{-1}=n-1$, showing that this choice of $\rho$ is extremal.
\end{example}

\begin{lemma}\label{lem:d-reg-mod}
    Let $G$ be an unweighted $d$-regular graph and let $\Gamma$ be the family of all stars in $G$. Then, for $1<p<\infty$, the density $\rho \equiv \frac{1}{d}$ is extremal and $\Mod_p(\Gamma) = \frac{|E|}{d^p}$.
\end{lemma}

The proof of Lemma~\ref{lem:d-reg-mod} is given in Section~\ref{sec:prob-interp} once the probabilistic interpretation of modulus has been developed.

\begin{example}[Cycle Graphs]\label{ex:cycle-star}
Let $G = C_n$ be the unweighted cycle graph. Since $G$ is a $2$-regular graph, by Lemma \ref{lem:d-reg-mod}, $\rho \equiv \frac{1}{2}$ is extremal and 
\[ 
\Mod_2(\gam{star}) = \frac{n}{2^2} = \frac{n}{4}.
\]
\end{example}

\begin{example}[Complete Graph]\label{ex:complete-star}
Let $G = K_n$ be the unweighted complete graph. Since $G$ is a $(n-1)$-regular graph, by Lemma \ref{lem:d-reg-mod}, $\rho \equiv \frac{1}{n-1}$ is extremal and 
\[ 
\Mod_2(\gam{star}) = \frac{\frac{n(n-1)}{2}}{(n-1)^2} = \frac{n}{2(n-1)}.
\]
\end{example}

By duality, Examples~\ref{ex:cycle-star} and~\ref{ex:complete-star} complete Examples~\ref{ex:cycle-fec} and~\ref{ex:complete-fec}.

\subsection{Probabilistic Interpretation}\label{sec:prob-interp}

The optimization problem in \eqref{eq:optimization_prob} is a convex problem for which, for $1 < p < \infty$, the minimizer is unique. A Lagrangian dual problem was developed in~\cite{albin2015modulus} and endowed with a probabilistic interpretation in~\cite{albin2016prob}. We review some of the concepts of this probabilistic interpretation and apply it to the families of stars and edge covers. 

Let $\mathcal{P}(\Gamma)\subset\mathbb{R}^{\Gamma}_{\ge 0}$ represent the set of probability mass functions (pmfs) on the set $\Gamma$. That is, $\mu \in \mathcal{P}(\Gamma)$ if and only if $\mu$ is a nonnegative vector with $\mu^T \mathbf{1} = 1$. 
For a given $\mu$, we can define a random variable $\rv{\gamma}$ with distribution given by $\mu: \mathbb{P}_{\mu} (\rv{\gamma} = \gamma) = \mu(\gamma)$. For an edge $e \in E$, the value $\cN(\rv{\gamma}, e)$ is a random variable as well and we denote its expectation as $\mathbb{E}_{\mu}[\cN(\rv{\gamma}, e)]$. The probabilistic interpretation is given by the following theorem. 

\begin{theorem}[Theorem 2 in \cite{albin2019blocking}]\label{thm:prob_int}
    Let $G=(V,E)$ be a finite graph with edge weights $\sigma$, and let $\Gamma$ be a non-trivial finite family of objects on $G$ with usage matrix $\cN$. Then, for any $1<p<\infty$, letting $q := p/(p-1)$ be the conjugate exponent to $p$, we have
    \begin{equation}\label{eq:mod_prob}
        \Mod_{p,\sigma}(\Gamma)^{-\frac{1}{p}} = \left( \min_{\mu \in \mathcal{P}(\Gamma)} \sum_{e\in E} \sigma(e)^{-\frac{q}{p}}  \mathbb{E}_{\mu}[\cN(\rv{\gamma}, e)]^q  \right)^{\frac{1}{q}}. 
    \end{equation} 

    Moreover, $\mu \in \mathcal{P}(\Gamma)$ is optimal for the right-hand side of \eqref{eq:mod_prob} if and only if
    \begin{equation}
        \mathbb{E}_{\mu}[\cN(\rv{\gamma}, e)] = \frac{\sigma(e) \rho^*(e)^{\frac{p}{q}}}{\Mod_{p,\sigma}(\Gamma)} \qquad \forall e \in E. 
    \end{equation}
    where $\rho^*$ is the unique extremal density for $\Mod_{p,\sigma}(\Gamma)$.
\end{theorem}

As stated in \cite{albin2019blocking}, the probabilistic interpretation is particularly informative when $p=2$, $\sigma \equiv 1$, and $\Gamma$ is a collection of subsets of edges, so that the usage matrix $\cN$ is as defined in \eqref{eq:natural-N}. In this case, the duality relation from Theorem \ref{thm:prob_int} can be written as 
\begin{equation}\label{eq:prob-interp-2-norm}
\Mod_2(\Gamma)^{-1} = \min_{\mu \in \mathcal{P}(\Gamma)} \mathbb{E}_{\mu} | \rv{\gamma} \cap \rv{\gamma'}|
\end{equation}
where $\rv{\gamma}$ and $\rv{\gamma}'$ are two independent random variables chosen according to the pmf $\mu$, and $| \rv{\gamma} \cap \rv{\gamma'}|$ is their \textit{overlap}, which is also a random variable. This implies that computing the 2-modulus is equivalent to finding a pmf $\mu$ that minimizes the expected overlap of two independent, identically distributed random objects. 

The expectations that appear in this section have special forms in the case of $\gam{star}$. In particular, if $e=\{x,y\}\in E$, then
\begin{equation}\label{eq:expect-star}
    \mathbb{E}_\mu[\mathcal{N}(\underline{\gamma},e)] = \mu(\delta(x)) + \mu(\delta(y)).
\end{equation}
Moreover, the expected overlap can be written as 
\begin{equation}\label{eq:exp-overlap}
\mathbb{E}_{\mu} | \rv{\gamma} \cap \rv{\gamma'}| = \sum_{v,v' \in V} | \delta(v) \cap \delta(v')| \mu(\delta(v)) \mu(\delta(v')) = \sum_{v \in V} \deg(v) \mu(\delta(v))^2 + \sum_{v \in V} \sum_{v' \sim v} \mu(\delta(v)) \mu(\delta(v')),
\end{equation}
which expresses the expected overlap as the sum of two terms. In the second term, the notation $v'\sim v$ in the inner summation indicates that the sum is taken over all neighbors $v'$ of $v$. Since the expected overlap is being minimized, the first term suggests that stars with higher degrees, i.e.\ stars with more edges, will be assigned smaller $\mu$ values than stars with a smaller number of edges. The second term acts to minimize the probabilities of neighboring vertices.

In the case of $\gam{star}$, selecting the uniform pmf yields a simple lower bound on the modulus.

\begin{lemma}\label{lem:simple-bound-prob}
    Let $\Gamma=\gam{star}$ and $1<p<\infty$. Then,
    \begin{equation*}
        \Mod_{p,\sigma}(\Gamma) \ge \left(\frac{|V|}{2}\right)^p
        \left(\sum_{e\in E}\sigma(e)^{-\frac{q}{p}}\right)^{-\frac{p}{q}}.
    \end{equation*}
    In particular, if $\sigma\equiv 1$, then
    \begin{equation}\label{eq:prob-lower-bound-unweighted}
        \Mod_p(\Gamma) \ge \left(\frac{|V|}{2|E|^{\frac{1}{q}}}\right)^p.
    \end{equation}
\end{lemma}

\begin{proof}
Define $\mu$ to be the uniform pmf on $\Gamma$. That is, $\mu(\delta(v)) = \frac{1}{|V|}$ for every $v \in V$. Then, the expected edge usage for every edge $e$ is 
\[
\mathbb{E}_{\mu}[\cN(\rv{\gamma}, e)] = \frac{2}{|V|}.
\]
Substituting this into equation~\eqref{eq:mod_prob} we get
\begin{align*}
    \Mod_{p,\sigma}(\Gamma)^{-\frac{1}{p}} & \le \left(  \sum_{e\in E} \sigma(e)^{-\frac{q}{p}}  \left( \frac{2}{|V|}\right)^q  \right)^{\frac{1}{q}} \\
    & \le \frac{2}{|V|} \left(  \sum_{e\in E} \sigma(e)^{-\frac{q}{p}}  \right)^{\frac{1}{q}} \\
    \Mod_{p,\sigma}(\Gamma) & \ge \left( \frac{2}{|V|} \right)^{-p}  \left(\sum_{e\in E}\sigma(e)^{-\frac{q}{p}}\right)^{-\frac{p}{q}}.
\end{align*}
which gives an upper bound to the modulus. Letting $\sigma \equiv 1$ yields the result from equation~\eqref{eq:prob-lower-bound-unweighted}.
\end{proof}

This allows us to prove Lemma~\ref{lem:d-reg-mod}.

\begin{proof}[Proof of Lemma~\ref{lem:d-reg-mod}]
    Lemma~\ref{lem:simple-bound-mod-star} provides an upper bound. The lower bound follows from Lemma~\ref{lem:simple-bound-prob}. Indeed, since $G$ is assumed to be $d$-regular, $|V| = 2|E|/d$. Substituting into~\eqref{eq:prob-lower-bound-unweighted} yields the bound
    \begin{equation*}
        \Mod_p(\Gamma) \ge \left(\frac{|E|^{1-\frac{1}{q}}}{d}\right)^p
        = \frac{|E|}{d^p}.
    \end{equation*}
\end{proof}

The following examples show optimal pmf's for several types of graphs. These examples highlight the balance of terms in~\eqref{eq:exp-overlap}.

\begin{example}\label{ex:star-prob}
Let $G = S_n$ be the star graph and define $\mu = 0$ on the center star and $\mu = \frac{1}{n-1}$ on the remaining $n-1$ stars. Then $\mu \in \mathcal{P}(\gam{star})$. Moreover, for any edge $e \in E$, the expected edge usage is 
\[ 
\mathbb{E}_{\mu}[\cN(\rv{\gamma}, e)] = \frac{1}{n-1},
\]
since each edge sees exactly one of the stars of degree one.

Using~\eqref{eq:mod_prob}, we obtain
\[
\Mod_2(\gam{star})^{-1} \le \sum_{e\in E}\left(\frac{1}{n-1}\right)^2 = (n-1)\left( \frac{1}{n-1}\right)^2 = \frac{1}{n-1}.
\]
From Example~\ref{ex:star-star} we know that $\Mod_2(\gam{star}) = n-1$, so this pmf is optimal.
\end{example}

\begin{example}\label{ex:cycle-prob}
Let $G = C_n$ be the cycle graph. By choosing the uniform pmf as in the proof of Lemma~\ref{lem:simple-bound-prob}, we obtain the lower bound
\begin{equation*}
    \Mod_2(\gam{star}) \ge \frac{n}{4}.
\end{equation*}
From Example~\ref{ex:cycle-star} we know that $\Mod_2(\gam{star}) = \frac{n}{4}$, so this pmf is optimal.
\end{example}

\begin{example}\label{ex:complete-prob}
Let $G = K_n$ be the complete graph. As in the previous example, the uniform pmf yields a lower bound,
\begin{equation*}
    \Mod_2(\gam{star}) \ge \frac{n}{2(n-1)},
\end{equation*}
which coincides with the modulus as computed in Example~\ref{ex:complete-star}, showing that the uniform pmf is optimal.
    
\end{example}

\subsection{More Examples}
In the following examples we compute $\Mod_2(\gam{star})$ on several unweighted, undirected graphs.

\begin{example}[Path Graphs]\label{ex:path-star}
     Let $G = P_n$ be the unweighted path graph with $|V| = n\ge 3$ and $|E| = n-1$.
     \begin{itemize}
        \item[a)] For $n=3$, the density $\rho_0 \equiv 1$ is admissible for $\gam{star}$. So,
        \[ 
        \Mod_2(\gam{star}) \le 2.
        \]
        For the lower bound, define $\mu(\delta(v)) = \frac{1}{2}$ on the two nodes with degree 1 and $\mu(\delta(u)) = 0$ on the node with degree 2. Using~\eqref{eq:mod_prob},
        \begin{align*}
            \Mod_2(\gam{star})^{-1} & \le 2 \left( \frac{1}{2} \right)^2 = \frac{1}{2}.
        \end{align*}
        Thus we have that $\Mod_2(\gam{star}) = 2$, $\rho_0$ is the extremal density, and $\mu$ is an optimal pmf.

        \item[b)] For $n=4$, let $\rho_0$ be defined as in Figure \ref{fig:P4_rho0}. One can check that this density is admissible for $\gam{star}$. So,
        \[ 
        \Mod_2(\gam{star}) \le 2.
        \]
        For the lower bound, similar to the $n=3$ case, define $\mu(\delta(v)) = \frac{1}{2}$ on the nodes with degree 1 and $\mu(\delta(u)) = 0$ for the nodes with degree 2. The expected edge usage is $1$ on the two outer edges and $0$ on the inner edge. By~\eqref{eq:mod_prob},
        \begin{align*}
            \Mod_2(\gam{star})^{-1} & \le 2 \left( \frac{1}{2} \right)^2 = \frac{1}{2}.
        \end{align*}
        Thus we have that $\Mod_2(\gam{star}) = 2$, $\rho_0$ is the extremal density, and $\mu$ is an optimal pmf.

        \item[c)] For $n \ge 5$ and $n$ odd, define $\rho_0$ as 
        \[ 
        \rho_0(e) = \begin{cases} 1 & \text{if $e$ is a pendant edge}, \\ \frac{1}{2} & \text{ otherwise}. \end{cases}
        \]
        This density is admissible for $\gam{star}$, so
         \[ 
         \Mod_2(\gam{star}) \le \mathcal{E}_2(\rho_0)= 2\cdot 1^2 + (n-3)\cdot \left(\frac{1}{2}\right)^2 = \frac{n+5}{4}.
         \]

        For the lower bound define $\mu = \frac{4}{n+5}$ on the stars with degree 1. On the stars with degree 2, assign $\mu$ to be 0 and $\frac{2}{n+5}$ alternately. (See Figure~\ref{fig:mu-path}.) One can verify that this is a pmf on the stars. Since no two stars with positive probability share an edge,~\eqref{eq:prob-interp-2-norm} and~\eqref{eq:exp-overlap} show that
        \begin{align*}
            \Mod_2(\gam{star})^{-1} & \le 2 \left( \frac{4}{n+5} \right)^2 + \frac{n-3}{2} \cdot 2 \left( \frac{2}{n+5} \right)^2 = \frac{4n+20}{(n+5)^2} = \frac{4}{n+5}.
        \end{align*}
        Thus we have that $\Mod_2(\gam{star}) = \frac{4}{n+5}$ and $\rho_0$ is the extremal density and $\mu$ is an optimal pmf.

         \item[d)] For $n \ge 6$ and even, define $\rho_0$ to be 1 on the pendant edges, and alternating between the values $\rho_1$ and $\rho_2$ on the remaining edges, where 
         \[
         \rho_1 = \frac{n-4}{2n-6} \qquad \rho_2 = \frac{n-2}{2n-6}.
         \]
         (See Figures~\ref{fig:P6_rho0} and~\ref{fig:P8_rho0}.) Note this density is admissible for $\gam{star}$. Thus,
         \[
         \Mod_2(\gam{star}) \le 2 \cdot 1^2 + \left(\frac{n}{2}-1\right) \left( \frac{n-4}{2n-6}\right)^2 + \left( \frac{n}{2}-2\right) \left( \frac{n-2}{2n-6} \right)^2 = \frac{n^2+2n-16}{2\left(2n-6\right)}.
         \]
         To obtain a lower bound, first enumerate the first half of the vertices $v_1, v_2, \ldots, v_k$, where $k = \frac{n}{2}$. For the stars centered at $v_i$, $1 \le i \le k$, define $\mu$ as follows:
         \[
         \mu(\delta(v_i)) = \begin{cases}
             \frac{2(2n-6)}{n^2+2n-16}, & i = 1 \\
             0, & i = 2 \\
             \frac{2n-2(i+1)}{n^2+2n-16}, & 2 < i \le k, i \text{ odd} \\
             \frac{2(i-2)}{n^2+2n-16}, & 2< i \le k, i \text{ even} \\
         \end{cases}
         \]
         Once the first $k$ stars have been assigned a $\mu$ value, assign the values of $\mu$ to the following $k$ vertices by mirroring the values: $\mu(\delta(v_{n+1-i}))=\mu(\delta(v_i))$ for $i=1,2,\ldots,k$. A straightforward (but long) calculation shows that the corresponding lower bound on modulus agrees with the upper bound above.
         \begin{figure}
             \centering
             \includegraphics[width=3.8in]{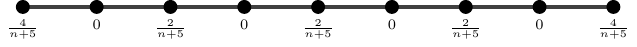}
             \caption{The values of $\mu$ for a path graph with $n \ge 5$ and odd.}
             \label{fig:mu-path}
         \end{figure}

    \end{itemize}

\begin{figure}
    \begin{subfigure}[h]{0.2\textwidth}
    \centering
    \includegraphics[width=0.8in]{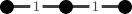}
    \caption{$P_3$}
    \label{fig:P3_rho0}
    \end{subfigure}
    \hfill
    \begin{subfigure}[h]{0.2\textwidth}
    \centering
    \includegraphics[width=1.2in]{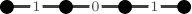}
    \caption{$P_4$}
    \label{fig:P4_rho0}
    \end{subfigure}
    \hfill
    \begin{subfigure}[h]{0.3\textwidth}
    \centering
    \includegraphics[width=1.5in]{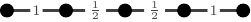}
    \caption{$P_5$}
    \label{fig:P5_rho0}
    \end{subfigure}
    \vfill
    \begin{subfigure}[h]{0.4\textwidth}
    \centering
    \includegraphics[width=2in]{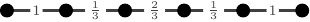}
    \caption{$P_6$}
    \label{fig:P6_rho0}
    \end{subfigure}
    \hfill
    \begin{subfigure}[h]{0.5\textwidth}
    \centering
   \includegraphics[width=2.4in]{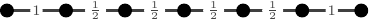}
    \caption{$P_7$}
    \label{fig:P7_rho0}
    \end{subfigure}
    \vfill
    \begin{subfigure}[h]{\textwidth}
    \centering
    \includegraphics[width=2.8in]{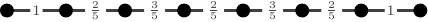}
    \caption{$P_8$}
    \label{fig:P8_rho0}
    \end{subfigure}
    \caption{The optimal density $\rho_0$ for path graphs when $p=2$.}
    \label{fig:path-extremal-density}
\end{figure}
\end{example}

\begin{example}[Wheel Graphs]\label{ex:wheel-star}
Let $G = W_n$ be the wheel graph with $|V| = n\ge 4$ and $|E| = 2(n-1)$. Specifically, $G$ has a special center node $v \in V$ with $\deg(v) = n-1$, and every other node $u \neq v$, has $\deg(u) = 3$.

\begin{itemize}
    \item[a)] For $4 \le n \le 5$, define $\rho_0$ as 
    \[
    \rho_0(e) = \begin{cases}
        \frac{1}{n-1}, & \text{if $e$ is incident to $v$} \\
        \frac{n-2}{2(n-1)}, & \text{otherwise}.
    \end{cases}
    \]
    See Figure~\ref{fig:wheel-extremal-density}. Note that $\rho_0$ is admissible for $\gam{star}$ and 
    \[
    \Mod_2(\gam{star}) \le (n-1)  \left(\frac{1}{n-1}\right)^2 + (n-1) \left(  \frac{n-2}{2(n-1)}\right)^2 = \frac{4+(n-2)^2}{4(n-1)}.
    \]
    To obtain a lower bound, define $\mu \in \mathcal{P}(\gam{star})$ as $\mu(\delta(v)) = \frac{6-n}{4+(n-2)^2}$ and define $\mu(\delta(u)) = \frac{n-2}{4+(n-2)^2}$ for $u \neq v$. Using \eqref{eq:expect-star} and Theorem~\ref{thm:prob_int},
    \begin{align*}
        \Mod_2(\gam{star})^{-1} & \le \left(\frac{6-n}{4+(n-2)^2}\right)^2 + 5(n-1) \left( \frac{n-2}{4+(n-2)^2} \right)^2 \\
        &+ 2 (n-1) \left(\frac{6-n}{4+(n-2)^2}\right) \left( \frac{n-2}{4+(n-2)^2} \right) \\
        & = \frac{4(n-1)(n^2-4n+8)}{(4+(n-2)^2)^2} = \frac{4(n-1)}{4+(n-2)^2} \\
    \end{align*}
    Thus we have that $\Mod_2(\gam{star}) = \frac{4+(n-2)^2}{4(n-1)}$, $\rho_0$ is the extremal density and $\mu$ is an optimal pmf. 

    \item[b)] For $n \ge 6$, define $\rho_0$ as 
    \[
    \rho_0(e) = \begin{cases}
        \frac{1}{5}, & \text{if $e$ is incident to $v$} \\
        \frac{2}{5}, & \text{otherwise}.
    \end{cases}
    \]
    Again, note that $\rho_0$ is admissible for $\gam{star}$ and 
    \[
    \Mod_2(\gam{star}) \le (n-1)  \left(\frac{1}{5}\right)^2 + (n-1) \left(  \frac{2}{5}\right)^2 = \frac{n-1}{5}.
    \]
    To obtain a lower bound, define $\mu(\delta(v)) = 0$ and $\mu(\delta(u)) = \frac{1}{n-1}$ for $u \neq v$. Using \eqref{eq:expect-star} and Theorem~\ref{thm:prob_int},
    \begin{align*}
        \Mod_2(\gam{star})^{-1} & \le 3(n-1)\left(\frac{1}{n-1}\right)^2 + 2(n-1)\left(\frac{1}{n-1}\right)^2 = \frac{5}{n-1}.
    \end{align*}
    Thus we have that $\Mod_2(\gam{star}) = \frac{n-1}{5}$, $\rho_0$ is the extremal density, and $\mu$ is an optimal pmf. 
\end{itemize}

\end{example}

\begin{figure}
        \begin{subfigure}{0.28\textwidth}
        \centering
        \includegraphics[width=1in]{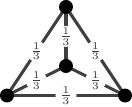}
        \caption{$W_4$}
        \label{fig:W4_rho0}
        \end{subfigure}
        \hfill 
        \begin{subfigure}{0.28\textwidth}
        \centering
        \includegraphics[width=1in]{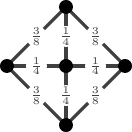}
        \caption{$W_5$}
        \label{fig:W5_rho0}
        \end{subfigure}
        \hfill 
        \begin{subfigure}{0.34\textwidth}
        \centering
        \includegraphics[width=1in]{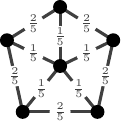}
        \caption{$W_6$.}
        \label{fig:W6_rho0}
        \end{subfigure}
        \caption{The extremal density $\rho_0$ for wheel graphs when $p=2$.}
        \label{fig:wheel-extremal-density}
    \end{figure}

\section{Comparing $\gam{ec}$ and $\gam{fec}$}\label{sec:num-examples}

For an unweighted graph, $G$, we can use the results from Lemma~\ref{lem:Bhat_fec} and Theorem~\ref{thm:mod_reciprocals} to restate the upper and lower bounds relating $\Mod_{2}(\gam{ec})$, $\Mod_2(\gam{fec})$ and $\Mod_{2}(\gam{star})$ as
\begin{equation}\label{eq:bounds_ec_star}
\left( \frac{3}{4} \right)^2 \leq 
\Mod_{2}(\gam{ec})\Mod_{2}(\gam{fec})^{-1} =
\Mod_{2}(\gam{ec})\Mod_{2}(\gam{star}) \leq 1,
\end{equation}
(If $G$ is bipartite, then $\gam{fec}\modeq\gam{ec}$, so the moduli of the two families are equal.) In this section, we present two examples that explore the tightness of the bounds in~\eqref{eq:bounds_ec_star} on nonbipartite graphs. We consider the complete graphs, $K_n$, for which we have explicit formulas for the moduli, and the $n$-barbell graphs, for which we use numerical approximation to compute the moduli.

\subsection{Numerical methods}\label{sec:num-methods}
One natural way to solve the modulus problem numerically is to apply a convex optimization solver to~\eqref{eq:optimization_prob}. For sufficiently small families, $\Gamma$, it is not difficult to form the full usage matrix $\mathcal{N}\in\mathbb{R}^{\Gamma\times E}$. For example, the usage matrix for $\gam{star}$ is a $V\times E$ matrix.

For larger families (e.g., $\gam{ec}$ and $\bfec$), it is often not feasible even to enumerate all constraints. A possible option for such cases is the \emph{basic algorithm} described in~\cite{albin2017modulus}. This algorithm proceeds by maintaining a relatively small set of active constraints, thus eliminating the need to fully construct $\mathcal{N}$. Instead, the algorithm iteratively improves an estimate of the optimal density. In each round of the iteration, the algorithm requires a subroutine \texttt{shortest($\rho$)}, that produces a $\rho$-shortest object in $\Gamma$, that is
\[
\gamma = \texttt{shortest}(\rho) \Longrightarrow \forall \gamma': \ell_{\rho}(\gamma) \le \ell_{\rho}(\gamma').
\]

As long as an efficient implementation of \texttt{shortest} exists for $\Gamma$, the basic algorithm can be used to numerically approximate modulus. For example, Dijkstra's algorithm can be used to compute the modulus of $st$-paths, and Kruskal's algorithm can be used to compute the modulus of spanning trees. This method can also be applied to the families in the present paper.

\paragraph{Stars.} For $\gam{star}$, it is possible to fully compute $\mathcal{N}$ (there is one row per vertex). There is also an efficient \texttt{shortest} method; one simply loops over all vertices to find the one with minimum $\rho$-degree.

\paragraph{Edge covers.}
For the family of edge covers, there is a minimum weight edge cover (MWEC) algorithm given by Schrijver in \cite{schrijver2003combinatorial}. This algorithm reduces the minimum weight edge cover problem to the minimum weight perfect matching (MWPM) problem, for which there are several polynomial-time algorithms~\cite{kuhn1955hungarian,edmonds1965paths,micali1980v}.

\paragraph{Fractional edge covers.}
The authors are not aware of an efficient method for computing the minimum weight fractional edge cover. (Considering the method described above for the MWEC problem, it is possible that an analogous reduction to the minimum weight fractional perfect matching problem exists.) Fortunately, it is possible to compute the modulus of this family anyway, using  Theorem~\ref{thm:mod_reciprocals} and Lemma~\ref{lem:Bhat_fec}. This allows us to repurpose the star modulus code to compute the modulus of fractional edge covers.

\paragraph{Implementation details.}
For the computations presented in this paper, we used the Python implementation of the basic algorithm found in~\cite{modbook}. The graphs are represented using NetworkX~\cite{SciPyProceedings_11}, NumPy~\cite{harris2020array}, and SciPy~\cite{2020SciPy-NMeth}, and the convex optimization problem is solved numerically using CVXPY~\cite{diamond2016cvxpy,agrawal2018rewriting}.  

When calculating the modulus of stars, it is efficient to compute the full usage matrix $\cN$ using the NetworkX built-in function \texttt{incidence\_matrix}. Since the family of edge covers tends to be large, it is not feasible to generate $\mathcal{N}$. Instead, we use the basic algorithm and implement the \texttt{shortest} subroutine with Schrijver's MWEC algorithm, using the NetworkX function \texttt{min\_weight\_matching} to solve the MWPM problem. It is also generally infeasible to generate the full usage matrix for fractional edge covers. For fractional edge covers, we calculate the modulus by using the results from Section~\ref{sec:fulkerson_duality} and the code that computes modulus of stars.


\subsection{Complete graphs}
The modulus values for $K_n$ were found in Examples~\ref{ex:complete-ec},~\ref{ex:complete-fec}, and~\ref{ex:complete-star}. Summarizing these examples, we found that
\begin{equation*}
    \Mod_2(\gam{ec}) =
    \begin{cases}
        \frac{2(n-1)}{n}\quad\text{ if $n$ is even,}\\
        \frac{2n(n-1)}{(n+1)^2}\quad\text{ if $n$ is odd},
    \end{cases}
    \quad\text{and}\quad
    \Mod_2(\gam{fec}) = \frac{2(n-1)}{n} = 
    \Mod_2(\gam{star})^{-1}.
\end{equation*}

In the case of edge covers, the value of the modulus splits into two cases depending on whether $n$ is even or odd. When $n$ is even, the minimum edge covers of the graph are also perfect matchings, therefore every node is covered by exactly one edge. When $n$ is odd, a minimum edge cover will have one node covered by 2 edges, and the rest covered by exactly one edge. Therefore, there will always be a ``heavier'' node in the edge cover. 

The modulus of fractional edge covers does not appear to have this same dependence on the evenness or oddness of $n$; there is a single formula for all cases. This suggests that, when $n$ is odd, the fractional edge covers using odd cycles become important. For this class of graphs, we have
\begin{equation*}
    \Mod_2(\gam{ec})\Mod_2(\gam{fec})^{-1} = 
    \begin{cases}
        1 & \text{if $n$ is even},\\
        \frac{n^2}{(n+1)^2}&\text{if $n$ is odd}.
    \end{cases}
\end{equation*}
Equation~\eqref{eq:bounds_ec_star} shows that the ratio of edge cover modulus to fractional edge cover modulus is bounded from below by $\frac{9}{16}$. The actual ratio approaches $1$ in the limit as $n\to\infty$, showing that the lower bound is not tight.

\begin{figure}
    \centering
    \includegraphics{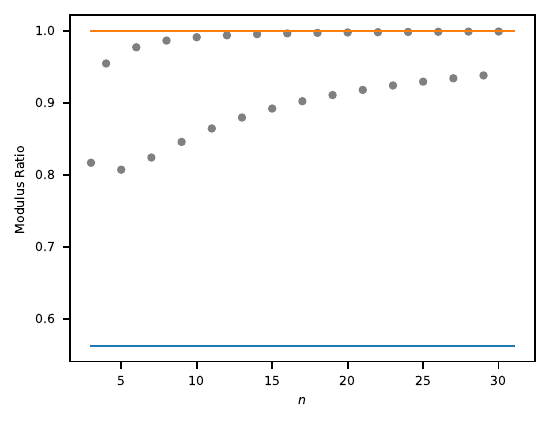}
    \caption{The value of $\Mod_2(\gam{ec})/\Mod_2(\gam{fec})$ as a function of $n$ for the $n$-barbell graph. The lines represent the upper and lower bounds in equation~\eqref{eq:bounds_ec_star}.}
    \label{fig:modulus-ratio-barbell}
\end{figure}

\subsection{Barbell graphs}
In the case of the $n$-barbell graphs (two disjoint copies of $K_n$ connected by a single edge), we see a similar behavior. Using the modulus code described in Section~\ref{sec:num-methods}, we computed $\Mod_2(\gam{ec})$ and $\Mod_2(\gam{fec})$ on $K_n$ for $n$ ranging from 3 to 30. The ratio of these two moduli is plotted in Figure~\ref{fig:modulus-ratio-barbell}. As in the complete graph example, we observe two types of behavior depending on the parity of $n$. In both cases, it appears that the ratio of moduli approaches $1$ as $n\to\infty$. However, the convergence is much faster for even $n$ than for odd. Again, this suggests that the odd cycles play an important role in $\Mod(\gam{fec})$ for odd $n$.

The difference can be further understood by considering the family of minimum edge covers on the $n$-barbell. When $n$ is even, no minimum edge cover uses the ``bridge'' of the $n$-barbell; these edge covers are built independently from edge covers of the two copies of $K_n$. From the perspective of the probabilistic interpretation, it appears that the optimal pmfs avoid using the bridge.

When $n$ is odd, however, all minimum edge covers use the bridge. If the optimal pmfs were to concentrate only on minimum edge covers, then the bridge would be ``overloaded.'' For these graphs, the optimal pmfs appear to balance betwen minimum edge covers and slightly larger edge covers that avoid the bridge. Choosing the smaller edge covers is beneficial since it tends to reduce the overlap with other edge covers; however, since these edge covers share a common edge (the bridge) it is also beneficial to choose the larger edge covers at times. An optimal pmf will balance these two competing preferences.

This can be seen in Figure~\ref{fig:barbells}. When $n$ is odd, the bridge is more likely to appear than any other edge in a random edge cover chosen by an optimal pmf (Figure~\ref{fig:barbells} top left). On the other hand, when $n$ is even, the optimal edge usage probability of the bridge is zero (Figure~\ref{fig:barbells} top right). For fractional edge covers, the bridge always has the lowest edge usage probability, followed by adjacent edges, and then all other edges.

Figure~\ref{fig:bridge-exp} compares the optimal expected edge usage, $ \mathbb{E}_{\mu^*}[\cN(\rv{\gamma}, b)]$, of the bridge, $b$, for both the edge cover and fractional edge cover modulus. While the expected usage for the fractional edge cover modulus follows a single smooth curve
\begin{equation*}
    \mathbb{E}_{\mu^*}[\cN(\rv{\gamma}, b)] = \frac{2n+2}{n^2-n+4},
\end{equation*}
the expected edge usage for edge cover modulus oscillates between two behaviors:
\begin{equation*}
    \mathbb{E}_{\mu^*}[\cN(\rv{\gamma}, b)] = 
\begin{cases}
    0 & \text{if $n$ is even},\\
    \frac{n-2}{n^2-n-1}&\text{if $n$ is odd}.
\end{cases}    
\end{equation*}
(There is a special case for $n=3$.)

From this example, we see that fractional edge cover modulus approximates edge cover modulus for large $n$. (Again, the lower bound in~\eqref{eq:bounds_ec_star} is overly pessimistic.) However, for smaller $n$, the additional flexibility in $\gam{fec}$ arising from the ability to use odd cycles, fails to capture the parity-dependence of the edge cover modulus.

\begin{figure}
    \centering
    \includegraphics[width=6in]{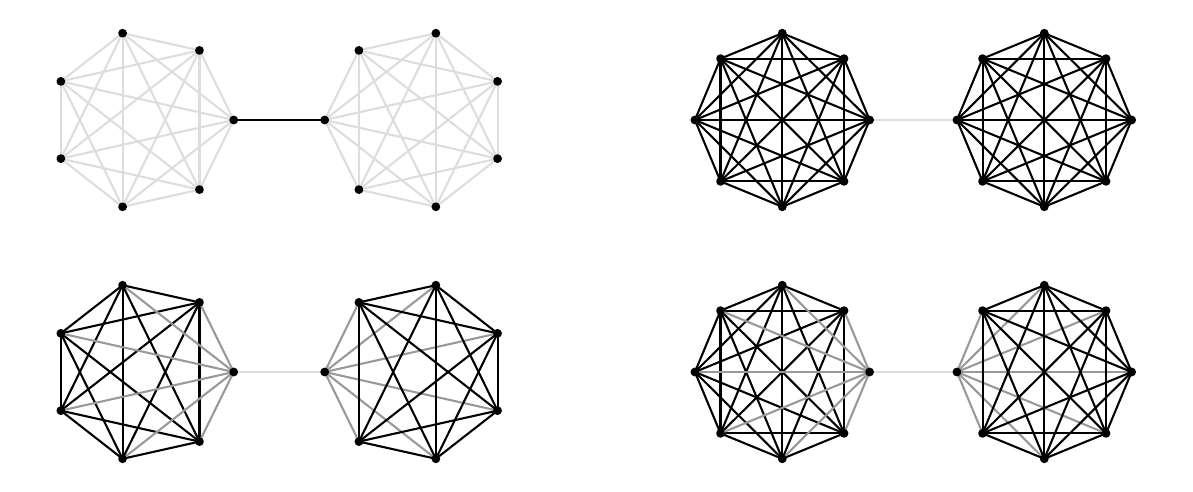}
    \caption{$n$-Barbell Graphs with $n=7$ and $n=8$. The edges are colored using the expected edge usage, $\mathbb{E}_{\mu^*}[\cN(\rv{\gamma}, e)] = \rho^*(e)/\Mod_2(\Gamma)$, for $\Gamma=\gam{ec}$ (top row) and $\Gamma=\gam{fec}$ (bottom row). Lighter colors represent smaller values.}
    \label{fig:barbells}
\end{figure}

\begin{figure}
    \centering
    \includegraphics[width=4.2in]{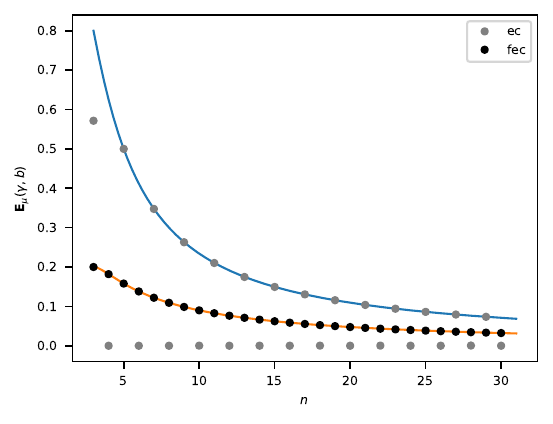}
    \caption{The graph above shows the value of $\mathbb{E}_{\mu}(\rv{\gamma}, b)$ in the $y$-axis and the value $n$ on the $x$-axis, where $b$ is the bridge connecting the two copies of $K_n$ in the $n$-barbell graph.}
    \label{fig:bridge-exp}
\end{figure}

\section{Discussion}\label{sec:discussion}

As mentioned in the introduction, one of the main motivations for studying the modulus of edge covers is to develop a deeper understanding of what properties of the underlying graph structure this modulus can expose. In this paper, we have laid the theoretical groundwork for the study of edge cover modulus. Moreover, by connecting it to the modulus of fractional edge covers and, ultimately, to that of stars, we have made it computationally feasible to approximate edge cover modulus on large graphs. In addition to deepening the connection between edge cover modulus and graph structure, there are several other interesting research directions open to pursuit.

One path involves further developing the relationship between the moduli of edge covers and fractional edge covers. For bipartite graphs, these two families are equivalent (in the sense of modulus), which gives a starting point. Moreover, if a graph only has very long odd cycles (a sort of ``nearly-bipartite'' property), then the bounds established in Theorem~\ref{thm:bounds} show that fractional edge cover modulus is a good approximation of edge cover modulus. However, the examples in Section~\ref{sec:num-examples} show that even for graphs containing triangles the two moduli can be close. In those examples, we saw that the approximation gets asymptotically better for larger graphs; the main barrier for the smaller examples seems to be a switching in the behavior of the edge cover modulus depending on whether or not the graph contains perfect matchings. The ability to use half-weight odd loops appears to give the fractional edge covers enough added flexibility to avoid this switching behavior.

Another question that arises naturally is that of the dual family to $\gam{ec}$. As shown in this paper, $\hat{\Gamma}_{\text{fec}}\modeq\gam{star}$, which leads to a computationally efficient method for computing the modulus of $\gam{fec}$. Computing the modulus of edge covers tends to be slower because of the need to repeatedly construct minimum edge covers as part of the basic algorithm. The edge covers may have a simpler dual family that could similarly aid in efficiently computing edge cover modulus.

Finally, we hope that a better understanding of the properties of the edge cover modulus will lead to similar insights into the modulus of related (and more complex) families, particularly the families of maximal matchings and perfect matchings. Both of these families have related relaxed (fractional) families that may be useful in their study just as the fractional edge covers can be used to understand edge covers.

\section*{Acknowledgements}
This material is based upon work supported by the National Science Foundation under Grant No.~2154032.

\printbibliography

\end{document}